\RequirePackage{fix-cm}
\documentclass[smallextended]{svjour3}       % onecolumn (second format)
\smartqed  % flush right qed marks, e.g. at end of proof

\usepackage{amsfonts}
\usepackage{epstopdf}
\usepackage{latexsym}
\usepackage{amsmath}
\usepackage{setspace}
\usepackage{comment}
\usepackage{bm}
\usepackage{graphics}
\usepackage{graphicx}
\usepackage[caption=false]{subfig}
\usepackage{color}
\usepackage{algorithm}
\usepackage{algpseudocode}
\usepackage{algorithmicx}
\usepackage{amsopn}
\usepackage{amssymb}
\usepackage{mathtools}

\newcommand{\mini}{\mathop{\rm minimize}}
\newcommand{\maxi}{\mathop{\rm maximize}}

\newcommand{\subj}{\mathop{\rm subject~to}}
\newcommand{\A}{{\cal A}}

\newcommand{\R}{\mathbb{R}}
\newcommand{\N}{{\cal N}}
\newcommand{\T}{{\cal T}}
\newcommand{\M}{{\cal M}}
\newcommand{\U}{{\cal U}}
\newcommand{\V}{{\cal V}}

\renewcommand{\S}{\mathbb{S}}
\renewcommand{\P}{{\cal P}}

\begin{document}

\title{
Superlinear and quadratic convergence of a stabilized sequential quadratic semidefinite programming method for nonlinear semidefinite programming problems
\thanks{
The author was in part supported by the Japan Society for the Promotion of Science KAKENHI for 21K17709.
}
}
% \subtitle{Do you have a subtitle?\\ If so, write it here}

\titlerunning{Local convergence of a stabilized SQSDP method for NSDP}         % if too long for running head

\author{Yuya Yamakawa}


\institute{
Yuya Yamakawa \at Graduate School of Informatics, Kyoto University, Yoshida-Honmachi, Sakyo-ku, Kyoto 606-8501, Japan,
\\
\email{yuya@i.kyoto-u.ac.jp}
}

\date{Received: date / Accepted: date}

\maketitle

\begin{abstract}
In this paper, we present a stabilized sequential quadratic semidefinite programming (SQSDP) method for nonlinear semidefinite programming (NSDP) problems and prove its local convergence. The stabilized SQSDP method is originally developed to solve degenerate NSDP problems and is based on the stabilized sequential programming (SQP) methods for nonlinear programming (NLP) problems. Although some SQP-type methods for NSDP problems have been proposed, most of them are SQSDP methods which are based on the SQP methods for NLP problems, and there are few researches regarding the stabilized SQSDP methods. In particular, there is room for the development of locally fast convergent stabilized SQSDP methods. We prove not only superlinear but also quadratic convergence of the proposed method under some mild assumptions, such as strict Robinson's constraint qualification and second-order sufficient condition.
%without the strict complementarity condition, constraint nondegeneracy, and strong second-order sufficient condition.
%Finally, we conduct some numerical experiments to confirm the performance of the proposed method.

\end{abstract}

\section{Introduction} \label{section:intro}
We consider solving the following nonlinear semidefinite programming (NSDP) problem:
\begin{align}
\begin{aligned} \label{NSDP}
& \mini_{x \in \R^{n}} & & f(x)
\\ 
& \subj & & g(x) = 0, ~ X(x) \succeq O,
\end{aligned}
\end{align}
where the functions $f \colon \R^{n} \rightarrow \R$, $g \colon \R^{n} \rightarrow \R^{m}$, and $X \colon \R^{n} \rightarrow \S^{d}$ are twice continuously differentiable, and $\R^{n}$ denotes the $n$-dimensional real space, and $\S^{d}$ represents the set of $d \times d$ real symmetric matrices. Moreover, let $\S_{++}^{d} ~ (\S_{+}^{d})$ be the set of $d \times d$ real symmetric positive (semi)definite matrices. For a matrix $M \in \S^{d}$, $M \succ O$ and $M \succeq O$ respectively indicate $M \in \S_{++}^{d}$ and $M \in \S_{+}^{d}$. 
\par
NSDP problems are known as a wide class optimization problems and have many applications such as robust control theory, structural optimization, finance, statistics, and so forth. Since it is desired to solve NSDP problems in various fields, a lot of researchers have proposed optimization methods to solve them. Nowadays, there exist many kinds of optimization methods, such as sequential quadratic semidefinite programming (SQSDP) methods~\cite{CoRa04,FaNoAp02,GoRa10,LiZh19,ZhCh16,ZhCh18,ZhZh14}, stabilized SQSDP methods~\cite{OkYaFu22,YaOk22}, primal-dual interior point methods~\cite{ArOkTa23,KaYaYa15,okuno2020local,YaYa14,YaYa15,HYaHYa12,YaYaHa12}, and augmented Lagrangian methods~\cite{AnHaVi20,HuTeYa06,KoSt03,SuSuZh08,WuLuDiCh13}. Most of them are mainly based on optimization methods for nonlinear programming (NLP) problems.
\par
Now we focus on the stabilized SQSDP methods for problem~\eqref{NSDP}. To introduce the stabilized SQSDP methods, we shall start by providing a brief review of the existing research regarding the SQSDP methods. They have been proposed by the existing literatures stated above and solve a sequence of quadratic SDP (QSDP) subproblems to generate a search direction and new Lagrange multipliers. Hence, we can regard the SQSDP methods as an extension of sequential quadratic programming (SQP) methods for NLP problems. On the other hand, the stabilized SQSDP methods are based on the stabilized SQP methods~\cite{GiKuRo17,gill2017stabilizedsup,GiRo13} for NLP problems and solve a sequence of stabilized QSDP subproblems instead of the QSDP subproblems. The stabilized subproblems have the following several advantages over the original ones: They always have at least one feasible solution and satisfy the Slater constraint qualification (CQ), and have a stabilizing effect on calculation related to Lagrange multipliers. In particular, from the second property of the stabilized subproblems, it is known that the stabilized SQSDP methods are effective for degenerate problems which do not satisfy CQs. Indeed, Yamakawa et al.~\cite{YaOk22} and Okabe~et al.~\cite{OkYaFu22} proposed the stabilized SQSDP methods for degenerate NSDP problems and proved its global convergence. However, there are few types of research regarding the stabilized SQSDP methods. Moreover, as far as we know, no previous research has proposed stabilized SQSDP methods with fast local convergence.
\par
The purpose of this paper is to propose a locally convergent stabilized SQSDP method and to prove its superlinear and quadratic convergence under some mild assumptions, such as the strict Robinson's constraint qualification (SRCQ) and second-order sufficient condition (SOSC). Regarding the existing SQSDP methods, several researchers have analyzed their local convergence under some rather strong assumptions such as the strict complementarity condition, constraint nondegeneracy, and strong SOSC, e.g., \cite{FaNoAp02,LiZh19,ZhCh16,ZhCh18}. More precisely, in the local convergence analysis of the existing literatures, the KKT conditions of the subproblem are basically regarded as the Newton equation, and local convergence is shown by utilizing the framework of the Newton method. Then one might consider that we can show local convergence of the stabilized SQSDP method by the same technique as the existing research. However, it is difficult to utilize such an existing technique of local convergence because of the different subproblem structures of the two methods. Hence, we analyze the convergence rate of the proposed method in another way.
\par
The organization of the paper is as follows. Section~\ref{section:pre} introduces some notation and terminologies required in the subsequent sections. In Section~\ref{section:stabilized_SQSDP}, we propose a locally convergent stabilized SQSDP method. In Section~\ref{section:local_convergence}, we prove fast local convergence of the proposed stabilized SQSDP method. 
%Section~\ref{section:numerical} reports some numerical results of the proposed method and examines its efficiency. 
Finally, we provide some conclusions and remarks in Section~\ref{section:conclusion}.
\par
We provide some mathematical notation. Let $\mathbb{N}$ be the set of natural numbers (positive integers). For $p \in \mathbb{N}$ and $q \in \mathbb{N}$, the set of matrices with $p$-rows and $q$-columns is denoted by $\R^{p \times q}$. Let $w \in \R^{p}$ and $W \in \R^{p \times q}$. The notation $[w]_{i}$ indicates the $i$-th element of $w$, and the notation $[W]_{ij}$ means the $(i,j)$-th entry of $W$. The transposition of $W$ is represented as $W^{\top} \in \R^{q \times p}$. We write $W^{\dagger} \in \R^{q \times p}$ to indicate the Moore-Penrose generalized inverse of $W$. If $W$ is a square matrix, then we denote by ${\rm tr}(M)$ the trace of $W$. For $V \in \R^{p \times q}$ and $W \in \R^{p \times q}$, the inner product of $V$ and $W$ is defined by $\left\langle V, W \right\rangle \coloneqq {\rm tr}(V^{\top}W)$, and the Hadamard product of $V$ and $W$ is written as $V \circ W \coloneqq ([V]_{ij}[W]_{ij}) \in \R^{p \times q}$. Note that if $q = 1$, then $\left\langle \cdot, \cdot \right\rangle$ is the inner product of vectors in $\R^{p}$. We denote by $I$ and $e$ the identity matrix and the all-ones vector, respectively, where these dimensions are defined by the context. For $w \in \R^{p}$, $\Vert w \Vert$ represents the Euclidean norm of $w$ defined by $\Vert w \Vert \coloneqq \sqrt{\left\langle w, w \right\rangle}$. For $W \in \R^{p \times q}$, $\Vert W \Vert_{{\rm F}}$ denotes the Frobenius norm of $W$ defined by $\Vert W \Vert_{{\rm F}} \coloneqq \sqrt{\left\langle W, W \right\rangle}$, and $\Vert W \Vert_{2}$ stands for the operator norm of $W$ defined by $\Vert W \Vert_{2} \coloneqq \sup \{ \Vert W u \Vert ; \Vert u \Vert \leq 1 \}$. We define the norm of product space ${\cal Z} \coloneqq \R^{p_{1} \times q_{1}} \times \R^{p_{2} \times q_{2}}$ as $\Vert z \Vert \coloneqq \Vert V \Vert_{{\rm F}} + \Vert W \Vert_{{\rm F}}$ for $z = (V, W) \in {\cal Z}$. Note that if $q_{1} = 1$, then the norm of ${\cal Z} = \R^{p_{1}} \times \R^{p_{2} \times q_{2}}$ is equivalent to $\Vert z \Vert = \Vert w \Vert + \Vert W \Vert_{{\rm F}}$ for $z = (w, W) \in {\cal Z}$. For $r_{1}, \ldots, r_{\ell} \in \R$, let us define ${\rm diag} \left[ r_{1}, \ldots, r_{\ell} \right]$ as the diagonal matrix whose $(i, i)$-th entry is equal to $r_{i}$. Let $U \in \S^{\ell}$ be a matrix with an orthogonal diagonalization $U = P D P^{\top}$, where $P$ is an orthogonal matrix and $D$ is a diagonal matrix. The eigenvalues of $U$ are indicated by $\lambda_{1}^{P}(U), \ldots, \lambda_{\ell}^{P}(U)$, and they satisfy $D = {\rm diag}[\lambda_{1}^{P}(U), \ldots, \lambda_{\ell}^{P}(U)]$. The minimum and the maximum eigenvalues of $U$ are respectively represented as $\lambda_{\min}(U)$ and $\lambda_{\max}(U)$. Let $\varphi$ be a function from $\R^{p}$ to $\R$. The gradient of $\varphi$ at $w \in \R^{p}$ is represented as $\nabla \varphi(w)$ or $\nabla_{w} \varphi(w)$. The Hessian of $\varphi$ at $w \in \R^{p}$ is indicated by $\nabla^{2} \varphi(w)$ or $\nabla_{ww}^{2} \varphi(w)$. Let ${\cal X}$ be a finite-dimensional real vector space equipped with an inner product $\langle \cdot, \cdot \rangle$ and its induced norm $\Vert \cdot \Vert$. Let ${\cal K} \subset {\cal X}$ and ${\cal L} \subset {\cal X}$ be arbitrary sets. We define ${\cal K} + {\cal L} \coloneqq \{ \eta + \mu; \eta \in {\cal K}, \mu \in {\cal L} \}$ and denote $\alpha {\cal K} \coloneqq \{ \alpha \eta; \eta \in {\cal K} \}$ for any $\alpha \in \R$. We write ${\rm conv}({\cal K})$ to indicate the convex hull of ${\cal K}$. If ${\cal K}$ is a non-empty set, then the tangent and normal cones of ${\cal K}$ at $\eta \in {\cal K}$ are respectively represented by $\T_{{\cal K}}(\eta)$ and $\N_{{\cal K}}(\eta)$. If ${\cal K}$ is a non-empty closed convex set, then $\P_{{\cal K}}(\eta)$ indicates the metric projection of $\eta \in {\cal X}$ on ${\cal K}$. For $\eta \in {\cal X}$ and $r > 0$, we define $B(\eta, r) \coloneqq \{ \mu \in {\cal X}; \Vert \mu - \eta \Vert \leq r \}$. We also use the notation $B \coloneqq B(0, 1)$ for simplicity. Let ${\cal F}$ and ${\cal G}$ be functions from ${\cal U}$ to ${\cal Y}$, where ${\cal U} \subset {\cal X}$ is an open set and ${\cal Y}$ is a finite-dimensional real vector space equipped with an inner product $\langle \cdot, \cdot \rangle$ and its induced norm $\Vert \cdot \Vert$. The Fr\'echet differential of ${\cal F}$ at $\eta \in {\cal U}$ is defined by ${\cal F}^{\prime}(\eta)$. The directional derivative of ${\cal F}$ at $\eta \in {\cal U}$ along $\mu \in {\cal U}$ is defined as ${\cal F}^{\prime}(\eta; \mu)$. We write ${\cal F}(\eta) = {\cal O}({\cal G}(\eta))~(\eta \to \mu)$ if there exists $r > 0$ and $c > 0$ such that $\Vert {\cal F}(\eta) \Vert \leq c \Vert {\cal G}(\eta) \Vert$ for all $\eta \in B(\mu, r)$.

\section{Preliminaries} \label{section:pre}
This section provides some notation and terminologies.

\subsection{Optimality conditions for NSDP}
To begin with, we define the following notation regarding the functions $g$ and $X$ in problem~\eqref{NSDP}:
\begin{itemize}
\item $g_{i} \colon \R^{n} \to \R$ denotes the $i$-th function of $g$, i.e., $g(x) = [g_{1}(x)  \cdots g_{m}(x)]^{\top}$;
\item $\nabla g(x) \in \R^{n \times m}$ means $\nabla g(x) \coloneqq [\nabla g_{1}(x) \cdots \nabla g_{m}(x)]$;
\item $A_{j}(x) \in \S^{d}$ indicates $A_{j}(x) \coloneqq \frac{\partial}{\partial [x]_{j}} X(x)$ for $j=1, \ldots, n$;
\item $\A(x) \colon \R^{n} \to \S^{d}$ is defined by $\A(x) u \coloneqq [u]_{1} A_{1}(x) + \cdots + [u]_{n} A_{n}(x)$ for all $u \in \R^{n}$;
\item the adjoint operator of $\A(x)$ is represented by $\A^{\ast}(x) \colon \S^{d} \to \R^{n}$, that is, $\A^{\ast}(x) U = [\langle A_{1}(x), U \rangle \cdots \langle A_{n}(x), U \rangle]^{\top}$ for all $U \in \S^{d}$.
\end{itemize}
Next, we define $\V \coloneqq \R^{n} \times \R^{m} \times \S^{d}$ and the Lagrange function $L \colon \V \to \R$ as
\begin{align*}
L(v) \coloneqq f(x) - \langle y, g(x) \rangle - \left\langle Z, X(x) \right\rangle,
\end{align*}
where $v \coloneqq (x, y, Z) \in \V$, and $y$ and $Z$ are Lagrange multipliers for $g(x) = 0$ and $X(x) \succeq O$, respectively. 
\par
Now we recall the definition of the Karush-Kuhn-Tucker (KKT) conditions for problem~\eqref{NSDP}. 
\begin{definition}
We say that $v = (x, y, Z) \in \V$ satisfies the KKT conditions if
\begin{align*} 
\displaystyle \nabla_{x} L(v) = 0, \quad g(x) = 0, \quad X(x) - \P_{\S_{+}^{d}}(X(x) - Z) = O.
\end{align*}
In particular, the above point $x$ is called a stationary point, and the triplet $v = (x, y, Z)$ is called a KKT point.
\end{definition}
The KKT conditions are known as the first-order optimality conditions for NSDP~\eqref{NSDP} under some CQ. Although there exist some CQs for NSDP~\eqref{NSDP}, we provide the SRCQ described below.
\begin{definition}
We say that a KKT point $(x,y,Z) \in \V$ satisfies the strict Robinson's constraint qualification if
\begin{align*}
\left[
\begin{array}{c}
\nabla g(x)^{\top}
\\
\A(x)
\end{array}
\right] \R^{n} + \left[
\begin{array}{c}
\{ 0 \}
\\
\T_{\S_{+}^{d}}(X(x)) \cap Z^{\perp}
\end{array}
\right] = \left[
\begin{array}{c}
\R^{m}
\\
\S^{d}
\end{array}
\right],
\end{align*}
where $Z^{\perp} \coloneqq \{ W \in \S^{d}; \langle W, Z \rangle = 0 \}$.
\end{definition}
Let $x^{\ast} \in \R^{n}$ be a stationary point of NSDP~\eqref{NSDP}. For the point $x^{\ast}$, we define the set $\M(x^{\ast})$ by
\begin{align*}
\M(x^{\ast}) \coloneqq \{ (y, Z) \in \R^{m} \times \S^{d}; \mbox{$(x^{\ast}, y, Z)$ satisfies the KKT conditions} \}.
\end{align*}
The SRCQ implies Robinson's CQ (RCQ) which is one of the well-known CQs and a sufficient condition under which $\M(x^{\ast})$ is nonempty and bounded. Moreover, the SRCQ is known as a sufficient condition for the uniqueness of the Lagrange multiplier pair, that is, $\M(x^{\ast})$ is a singleton set. For details of the facts, see~\cite[Section~5]{BoSh00}. In what follows, we assume that the SRCQ is satisfied at a KKT point $v^{\ast} \coloneqq (x^{\ast}, y^{\ast}, Z^{\ast})$, where $(y^{\ast}, Z^{\ast}) \in \Lambda(x^{\ast})$. We then notice that the assumption implies $\Lambda(x^{\ast}) = \{ (y^{\ast}, Z^{\ast}) \}$.
\par
Next, we provide the definition of the SOSC under the SRCQ. In this paper, we also suppose that the SOSC holds at $v^{\ast}$.
\begin{definition}
Let $v = (x, y, Z) \in \V$ be a KKT point of the NSDP~\eqref{NSDP} and let $v$ satisfy the SRCQ. We say that the second-order sufficient condition holds at $v$ if
\begin{align*}
\langle ( \nabla^{2}_{xx} L(v) + {\cal H}(x, Z) ) d, d \rangle > 0 \quad \forall d \in C(x) \backslash \{ 0 \},
\end{align*}
where the $(i,j)$-entry of the matrix ${\cal H}(x, Z)$ is given by
\begin{align*}
[{\cal H}(x, Z)]_{ij} \coloneqq 2 \langle Z, A_{i}(x)X(x)^{\dagger}A_{j}(x) \rangle,
\end{align*}
and the critical set $C(x)$ is defined as
\begin{align*}
C(x) \coloneqq \left\{ d \in \R^{n}; \nabla f(x)^{\top} d = 0, \nabla g(x)^{\top} d = 0, \A(x)d \in \T_{\S_{+}^{d}}(X(x)) \right\}.
\end{align*}
\end{definition}
\par
Note that the matrices $X(x^{\ast})$ and $Z^{\ast}$ satisfy $X(x^{\ast})Z^{\ast} = Z^{\ast} X(x^{\ast}) = O$ because the point $v^{\ast}$ satisfies the KKT conditions. Without loss of generality, they can be simultaneously diagonalized as follows.
\begin{align}
X(x^{\ast}) = P D P^{\top}, \quad Z^{\ast} = P \Lambda P^{\top}, \label{diagonalize:matXZ}
\end{align}
where $D$ and $\Lambda$ are some diagonal matrices given by
\begin{align*}
D = \left[
\begin{array}{ccc}
D_{\alpha} & O & O
\\
O & O & O
\\
O & O & O
\end{array}
\right], \quad D_{\alpha} \succ O, \quad \Lambda = \left[
\begin{array}{ccc}
O & O & O
\\
O & O & O
\\
O & O & \Lambda_{\gamma}
\end{array}
\right], \quad \Lambda_{\gamma} \succ O,
\end{align*}
and $P$ is some orthogonal matrix. Let $M^{\ast} \coloneqq X(x^{\ast}) - Z^{\ast}$, $\alpha \coloneqq \{ i \in \mathbb{N}; \lambda_{i}(M^{\ast}) > 0 \}$, $\beta \coloneqq \{ i \in \mathbb{N}; \lambda_{i}(M^{\ast}) = 0 \}$, and $\gamma \coloneqq \{ i \in \mathbb{N}; \lambda_{i}(M^{\ast}) < 0 \}$. Moreover, we define $P_{\alpha} \in \R^{d \times |\alpha|}$, $P_{\beta} \in \R^{d \times |\beta|}$, and $P_{\gamma} \in \R^{d \times |\gamma|}$ as three submatrices satisfying
\begin{align}
P = [ P_{\alpha} ~ P_{\beta} ~ P_{\gamma} ], \label{def:matP}
\end{align}
where notice that the cardinality of a set $S$ is represented as $|S|$. Using the above submatrices, the tangent cone $\T_{\S_{+}^{d}}(X(x^{\ast}))$ is given by
\begin{align}
\T_{\S_{+}^{d}}(X(x^{\ast})) = \left\{ M \in \S^{d}; [P_{\beta} ~ P_{\gamma}]^{\top} M [P_{\beta} ~ P_{\gamma}] \succeq O \right\}. \label{property:TSX}
\end{align}
For the details, see~\cite{Su06}.

\subsection{Perturbed problem related to NSDP}
For a given perturbed parameter $u \coloneqq (r,s,T) \in \V$, we define the following perturbed optimization problem related to~\eqref{NSDP}.
\begin{align}
\begin{aligned} \label{perturbedNSDP}
& \mini_{x \in \R^{n}} & & f(x) + \langle r, x \rangle
\\ 
& \subj & & g(x) + s = 0, ~ X(x) + T \succeq O.
\end{aligned}
\end{align}
The KKT conditions of \eqref{perturbedNSDP} are provided below.
\begin{align}
\nabla_{x} L(v) + r = 0, ~\, g(x) + s = 0, ~\, X(x) + T - P_{\S_{+}^{d}}( X(x) + T - Z ) = O. \label{perturbedKKT}
\end{align}
We denote by $\U(x, u)$ the set of Lagrange multipliers associated with $(x, u) \in \R^{n} \times \V$, i.e.,
\begin{align*}
\U(x, u) \coloneqq \{ (y, Z) \in \R^{m} \times \S^{d}; \mbox{$(x, y, Z)$ satisfies~\eqref{perturbedKKT} with $u = (r,s,T)$} \}.
\end{align*}
Note that $\M(x^{\ast}) = \U(x^{\ast}, 0)$.

\section{Stabilized SQSDP method} \label{section:stabilized_SQSDP}
The purpose of the section is to propose a locally convergent stabilized SQSDP method for problem~\eqref{NSDP}. Before providing the method, we first explain a brief outline of the SQSDP methods.
\par
The SQSDP methods are one kind of the extension of the SQP methods to solve NLP problems and they sequentially solve QSDP subproblems to find a search direction and the new Lagrange multipliers at each iteration. In particular, for a given iterate $v = (x, y, Z) \in \V$, the QSDP subproblem is provided as
\begin{align} \label{subproblem:SQSDP}
\begin{aligned}
& \mini_{\xi \in \R^{n}} & & \langle \nabla f(x), \xi \rangle + \frac{1}{2} \langle H(v) \xi, \xi \rangle
\\
& \subj & & g(x) + \nabla g(x)^{\top} \xi = 0, ~ X(x) + \A(x) \xi \succeq O,
\end{aligned}
\end{align}
where $H(v) \in \S^{n}$ is the Hessian matrix $\nabla_{xx}^{2} L(v)$ or its approximation satisfying $H(v) \to \nabla_{xx}^{2}L(v^{\ast})$ as $v \to v^{\ast}$. As stated above, the SQSDP methods find the KKT point $(\overline{\xi}, \overline{\zeta}, \overline{\Sigma}) \in \V$ of \eqref{subproblem:SQSDP}, adopt $\overline{\xi} \in \R^{n}$ as a search direction, and set $(\overline{\zeta}, \overline{\Sigma} ) \in \R^{m} \times \S_{+}^{d}$ as a new Lagrange multiplier pair.
\par
Now we recall that QSDP subproblem~\eqref{subproblem:SQSDP} is derived from the following minimax problem:
\begin{align} \label{problem:minimax_L}
\begin{aligned}
& \mini_{x \in \R^{n}} \maxi_{(y,Z) \in \R^{m} \times \S_{+}^{d}} & & L(x, y, Z).
\end{aligned}
\end{align}
Since solving \eqref{problem:minimax_L} is equivalent to finding the KKT point of \eqref{NSDP}, one might consider the following approximation of \eqref{problem:minimax_L} at $v \in \V$:
\begin{align} \label{problem:appro_minimax}
\begin{array}{cl}
\displaystyle \mini_{\xi \in \R^{n}} \maxi_{(\zeta, \Sigma) \in \R^{m} \times \S_{+}^{d}} & \displaystyle \langle \nabla f(x), \xi \rangle + \frac{1}{2} \langle H(v) \xi, \xi \rangle 
\\
& \quad - \langle \zeta, g(x) + \nabla g(x)^{\top} \xi \rangle - \langle \Sigma, X(x) + \A(x) \xi \rangle.
\end{array}
\end{align}
Then notice that \eqref{problem:appro_minimax} is equivalent to \eqref{subproblem:SQSDP}, and hence we can see the validity that the SQSDP methods iteratively solve subproblem~\eqref{subproblem:SQSDP}. 
\par
On the other hand, the stabilized SQSDP methods adopt the following approximation of \eqref{problem:minimax_L}:
\begin{align} \label{problem:stab_appro_minimax_L}
\begin{array}{cl}
\displaystyle \mini_{\xi \in \R^{n}} \maxi_{(\zeta, \Sigma) \in \R^{m} \times \S_{+}^{d}} 
& \displaystyle \langle \nabla f(x), \xi \rangle + \frac{1}{2} \langle H(v) \xi, \xi \rangle 
\\
& \quad - \langle \zeta, g(x) + \nabla g(x)^{\top} \xi \rangle - \langle \Sigma, X(x) + \A(x) \xi \rangle \vspace{2.5mm}
\\
& \quad \displaystyle - \frac{\sigma}{2} \Vert \zeta - y \Vert^{2} - \frac{\sigma}{2} \Vert \Sigma - Z \Vert_{{\rm F}}^{2},
\end{array}
\end{align}
where $\sigma > 0$ is a penalty parameter. The maximization part in \eqref{problem:stab_appro_minimax_L} has a unique optimum for each $\xi \in \R^{n}$ because of the  quadratic terms $\frac{\sigma}{2} \Vert \zeta - y \Vert^{2}$ and $\frac{\sigma}{2} \Vert \Sigma - Z \Vert_{{\rm F}}^{2}$. Furthermore, these quadratic terms have calming effect on a sequence of the Lagrange multipliers $\{ (y_{k}, Z_{k} ) \} \subset \R^{m} \times \S^{d}$, and we can easily confirm that solving \eqref{problem:stab_appro_minimax_L} is equivalent to finding the KKT point of 
\begin{align} \label{subproblem:stabilizedSQSDP}
\begin{aligned}
& \mini_{(\xi, \zeta, \Sigma) \in \V} & & \langle \nabla f(x), \xi \rangle + \frac{1}{2} \langle H(v) \xi, \xi \rangle + \frac{\sigma}{2} \Vert \zeta \Vert^{2} + \frac{\sigma}{2} \Vert \Sigma \Vert_{{\rm F}}^{2}
\\
& \subj & & g(x) + \nabla g(x)^{\top} \xi + \sigma (\zeta - y) = 0,
\\
& & & X(x) + \A(x) \xi + \sigma (\Sigma - Z) \succeq O.
\end{aligned}
\end{align}
From the calming effect stated above, this is called the stabilized QSDP subproblem. Note that this problem is always feasible whereas the standard subproblem~\eqref{subproblem:SQSDP} does not satisfy such a property.
\par
Finally, we propose a locally convergent stabilized SQSDP method. To this end, let us define the function $\sigma \colon \V \to \R$ as
\begin{align}
\sigma(v) \coloneqq \Vert \nabla_{x} L(v) \Vert + \Vert g(x) \Vert + \Vert X(x) - \P_{\S_{+}^{d}}(X(x) - Z) \Vert_{{\rm F}}. \label{def:sigma}
\end{align}
This function indicates the distance between a given point $v$ and the KKT point $v^{\ast}$, namely, $\sigma(v^{\ast}) = 0$, and is utilized as the penalty parameter $\sigma$ appeared in stabilized subproblem~\eqref{subproblem:stabilizedSQSDP}. The proposed stabilized SQSDP method is as follows:
\begin{algorithm}[tbh]
\caption{(Locally convergent stabilized SQSDP method)} \label{Local_SQSDP}
\begin{algorithmic}[1]
\State{Choose $v_{0} \coloneqq (x_{0}, y_{0}, Z_{0}) \in \V$. Set $k \coloneqq 0$ and $\sigma_{0} \coloneqq \sigma(v_{0})$.} \Comment{Step~0}

\State{If $v_{k}$ satisfies the termination criterion, then stop.} \Comment{Step~1}

\State{Set $H_{k} \coloneqq H(v_{k})$ and $\sigma_{k} \coloneqq \sigma(v_{k})$. Solve the stabilized subproblem} 
\begin{align*}
\begin{aligned}
& \mini_{(\xi, \zeta, \Sigma) \in \V} & & \langle \nabla f(x_{k}), \xi \rangle + \frac{1}{2} \langle H_{k} \xi, \xi \rangle + \frac{\sigma_{k}}{2} \Vert \zeta \Vert^{2} + \frac{\sigma_{k}}{2} \Vert \Sigma \Vert_{{\rm F}}^{2}
\\
& \subj & & g(x_{k}) + \nabla g(x_{k})^{\top} \xi + \sigma_{k} (\zeta - y_{k}) = 0,
\\
& & & X(x_{k}) + \A(x_{k}) \xi + \sigma_{k} (\Sigma - Z_{k}) \succeq O,
\end{aligned}
\end{align*}
and obtain its solution $(\overline{\xi}, \overline{\zeta}, \overline{\Sigma}) \in \V$. \Comment{Step~2}

\State{Set $v_{k+1} \coloneqq (x_{k} + \overline{\xi}, \overline{\zeta}, \overline{\Sigma})$ and $\sigma_{k+1} \coloneqq \sigma(v_{k+1})$.} \Comment{Step~3}

\State{Update the iteration as $k \leftarrow k+1$ and go back to Step~1.} \Comment{Step~4}

\end{algorithmic}
\end{algorithm}
\begin{remark}
As stated in Section~\ref{section:intro}, there are some existing studies associated with the local convergence of the SQSDP methods~\cite{FaNoAp02,LiZh19,ZhCh16,ZhCh18}. Their convergence analyses exploit the Newton equation derived from the KKT conditions of subproblem~\eqref{subproblem:SQSDP}. However, the convergence analyses of the SQSDP methods cannot be utilized in Algorithm~\ref{Local_SQSDP} because the structures of subproblems~\eqref{subproblem:SQSDP} and \eqref{subproblem:stabilizedSQSDP} are different. Moreover, the convergence analyses of the existing papers~\cite{gill2017stabilizedsup,Ha99,Wr98} regarding the stabilized SQP methods for NLP problems are also based on the Newton equation related to the stabilized subproblem which is obtained by using the active set of the inequality constraints. Hence, it is difficult to extend those ways into the analysis of this paper. In the following analysis regarding subproblem~\eqref{subproblem:stabilizedSQSDP}, we provide a new way of using a perturbed problem associated with \eqref{NSDP}. One of the remarkable points of this paper is to prove fast local convergence of Algorithm~\ref{Local_SQSDP} based on the new analysis which overcomes the difficulty related to the stabilized subproblem.
\end{remark}

\section{Local convergence of Algorithm~\ref{Local_SQSDP}} \label{section:local_convergence}
In this section, we prove fast local convergence of Algorithm~\ref{Local_SQSDP}. To this end, we first make the following assumptions:
\begin{description}
\item[(A1)] The SRCQ holds at $v^{\ast}$;
\item[(A2)] the SOSC holds at $v^{\ast}$;
\item[(A3)] the second derivatives of $f$, $g$, and $X$ are locally Lipschitz continuous on some neighborhood of $x^{\ast}$.
\end{description}
In addition to the above assumptions, we suppose that Algorithm~\ref{Local_SQSDP} generates an infinite sequence $\{ v_{k} \} \subset \V$ such that $v_{k} \not = v^{\ast}$ for every $k \in \mathbb{N} \cup \{ 0 \}$, and suppose that $H(v) \to \nabla_{xx}^{2} L(v^{\ast})$ as $v \to v^{\ast}$, where recall that $H(v) \in \S^{n}$ is the Hessian matrix $\nabla_{xx}^{2} L(v)$ or its approximation. Throughout the section, we frequently use the notation $M^{\ast} = X(x^{\ast}) - Z^{\ast}$.
We emphasize that the strict complementarity condition, constraint nondegeneracy, and strong SOSC are not assumed in the following analysis whereas all or some of them were required in the existing papers~\cite{LiZh19,ZhCh16,ZhCh18} which focus on local convergence of the SQSDP methods.
\par
Under assumptions (A1) and (A2), the function $\sigma$ defined by~\eqref{def:sigma} enjoys the error bound property. Before showing this property, let us provide the following helpful lemma.

\begin{lemma} \label{lemma:base}
Suppose that {\rm (A1)} and {\rm (A2)} are satisfied. If $(\widehat{x}, \widehat{y}, \widehat{Z}) \in \V$ satisfies
\begin{align*}
& \nabla_{xx}^{2} L(v^{\ast}) \widehat{x} - \nabla g(x^{\ast}) \widehat{y} - \A^{\ast}(x^{\ast}) \widehat{Z} = 0,
\\
& \nabla g(x^{\ast})^{\top} \widehat{x} = 0,
\\
& \A(x^{\ast}) \widehat{x} - {\cal P}_{\S_{+}^{d}}^{\prime}(X(x^{\ast}) - Z^{\ast}; \A(x^{\ast}) \widehat{x} - \widehat{Z}) = O,
\end{align*}
then $(\widehat{x}, \widehat{y}, \widehat{Z}) = (0,0,O)$.
\end{lemma}

\begin{proof}
We begin by showing
\begin{align}
- \langle \A(x^{\ast}) \widehat{x}, \widehat{Z} \rangle = \langle {\cal H}(x^{\ast}, Z^{\ast}) \widehat{x}, \widehat{x} \rangle. \label{eq:AZHx}
\end{align}
Let us define $\widehat{N} \coloneqq \A(x^{\ast}) \widehat{x} - \widehat{Z}$ for simplicity. Since $\A(x^{\ast}) \widehat{x} = \P_{\S_{+}^{d}}^{\prime}(M^{\ast}; \widehat{N})$ holds, it follows from \cite[Equation~(16)]{Su06} that
\begin{align} \label{eq:Proj2}
P^{\top} (\A(x^{\ast})\widehat{x}) P = \left[
\begin{array}{ccc}
P_{\alpha}^{\top} \widehat{N} P_{\alpha} & P_{\alpha}^{\top} \widehat{N} P_{\beta} & U_{\alpha \gamma} \circ P_{\alpha}^{\top} \widehat{N} P_{\gamma}
\\
P_{\beta}^{\top} \widehat{N} P_{\alpha} & {\cal P}_{\S_{+}^{|\beta|}}(P_{\beta}^{\top} \widehat{N} P_{\beta}) & O
\\
U_{\alpha \gamma}^{\top} \circ P_{\gamma}^{\top} \widehat{N} P_{\alpha} & O & O
\end{array}
\right],
\end{align}
where notice that $P$ satisfies \eqref{def:matP}, $U \coloneqq [u_{ij}] \in \S^{d}$ is defined by
\begin{align*}
u_{ij} \coloneqq \left\{
\begin{array}{lll}
1 & ~ & {\rm if} ~ i, j \in \beta,
\\
\displaystyle \frac{\max \{ \lambda_{i}(M^{\ast}), 0 \} + \max \{ \lambda_{j}(M^{\ast}), 0 \}}{|\lambda_{i}(M^{\ast})| + |\lambda_{j}(M^{\ast})|} & ~ & {\rm otherwise},
\end{array}
\right.
\end{align*}
and $U_{\alpha \gamma} \coloneqq [u_{ij}]_{i \in \alpha, j \in \gamma} \in \R^{|\alpha| \times |\gamma|}$. Moreover, each block matrix of~\eqref{eq:Proj2} implies
\begin{gather}
P_{\beta}^{\top} (\A(x^{\ast}) \widehat{x}) P_{\gamma} = O, \, P_{\gamma}^{\top} (\A(x^{\ast}) \widehat{x}) P_{\gamma} = O, \, P_{\alpha}^{\top} \widehat{Z} P_{\alpha} = O, \, P_{\alpha}^{\top} \widehat{Z} P_{\beta} = O, \label{eq:AZblock}
\\
P_{\beta}^{\top} (\A(x^{\ast}) \widehat{x}) P_{\beta} = \P_{\S_{+}^{|\beta|}} (P_{\beta}^{\top} \widehat{N} P_{\beta}), \label{eq:APZ}
\\
[P^{\top} \widehat{Z} P]_{ij} = \lambda_{j}(M^{\ast}) \lambda_{i}(M^{\ast})^{-1} [P^{\top} (\A(x^{\ast})\widehat{x}) P]_{ij} \quad \forall (i,j) \in \alpha \times \gamma. \label{eq:ZAblock}
\end{gather}
By exploiting the well-known property regarding $\P_{\S_{+}^{|\beta|}}$, it can be easily verified that $\langle P_{\beta}^{\top} \widehat{N} P_{\beta} - \P_{\S_{+}^{|\beta|}} (P_{\beta}^{\top} \widehat{N} P_{\beta}), S - \P_{\S_{+}^{|\beta|}} (P_{\beta}^{\top} \widehat{N} P_{\beta}) \rangle \leq 0$ for every $S \in \S_{+}^{|\beta|}$. Therefore, substituting $S = O$ and $S = 2\P_{\S_{+}^{|\beta|}} (P_{\beta}^{\top} \widehat{N} P_{\beta})$ into the inequality and using~\eqref{eq:APZ} yield ${\rm tr}[P_{\beta}^{\top} (\A(x^{\ast})\widehat{x}) P_{\beta} P_{\beta}^{\top} \widehat{Z} P_{\beta}] = \langle P_{\beta}^{\top} (\A(x^{\ast})\widehat{x}) P_{\beta}, P_{\beta}^{\top} \widehat{Z} P_{\beta} \rangle = 0$. It then follows from~\eqref{eq:AZblock} and~\eqref{eq:ZAblock} that 
\begin{align}
\begin{aligned} \label{eq:AZtrace}
-\langle \A(x^{\ast}) \widehat{x}, \widehat{Z} \rangle 
&= -2{\rm tr}[P_{\alpha}^{\top} (\A(x^{\ast})\widehat{x}) P_{\gamma} P_{\gamma}^{\top} \widehat{Z} P_{\alpha}] 
\\
& \qquad \qquad \qquad \qquad - {\rm tr}[P_{\beta}^{\top} (\A(x^{\ast})\widehat{x}) P_{\beta} P_{\beta}^{\top} \widehat{Z} P_{\beta}]
\\
&= -2 \sum_{i \in \alpha} \sum_{j \in \gamma} [P^{\top}(\A(x^{\ast})\widehat{x})P]_{ij} [P^{\top} \widehat{Z} P]_{ij}
\\
&= -2 \sum_{i \in \alpha} \sum_{j \in \gamma} \lambda_{j}(M^{\ast}) \lambda_{i}(M^{\ast})^{-1} [P^{\top}(\A(x^{\ast})\widehat{x})P]_{ij}^{2}.
\end{aligned}
\end{align}
Meanwhile, noting \eqref{diagonalize:matXZ} derives that $\lambda_{i}(M^{\ast}) = [D]_{ii}$ for $i \in \alpha$ and $\lambda_{i}(M^{\ast}) = -[\Lambda]_{ii}$ for $i \in \gamma$. Then, by the definition of ${\cal H}(x^{\ast}, Z^{\ast})$, we have
\begin{align}
\begin{aligned} \label{eq:HxZ}
\langle {\cal H}(x^{\ast}, Z^{\ast}) \widehat{x}, \widehat{x} \rangle
&= 2{\rm tr}[ \Lambda P^{\top} (\A(x^{\ast})\widehat{x}) P D^{\dagger} P^{\top} (\A(x^{\ast})\widehat{x}) P]
\\
&= 2{\rm tr}[ \Lambda_{\gamma} P_{\gamma}^{\top}(\A(x^{\ast})\widehat{x}) P_{\alpha} D_{\alpha}^{-1} P_{\alpha}^{\top} (\A(x^{\ast})\widehat{x}) P_{\gamma} ]
\\
&= -2 \sum_{i \in \alpha} \sum_{j \in \gamma} \lambda_{j}(M^{\ast}) \lambda_{i}(M^{\ast})^{-1} [P^{\top}(\A(x^{\ast})\widehat{x})P]_{ij}^{2}.
\end{aligned}
\end{align}
Combining~\eqref{eq:AZtrace} and ~\eqref{eq:HxZ} implies that~\eqref{eq:AZHx} holds. Note that $\nabla_{xx}^{2} L(v^{\ast}) \widehat{x} - \nabla g(x^{\ast}) \widehat{y} - \A^{\ast}(x^{\ast}) \widehat{Z} = 0$ and $\nabla g(x^{\ast})^{\top} \widehat{x} = 0$ are assumed. It then follows from~\eqref{eq:AZHx} that
\begin{align}
\begin{aligned} \label{ineq:nabLH}
0
&= \langle \nabla_{xx}^{2} L(v^{\ast}) \widehat{x} - \nabla g(x^{\ast}) \widehat{y} - \A^{\ast}(x^{\ast}) \widehat{Z}, \widehat{x} \rangle
\\
&= \langle \nabla_{xx}^{2} L(v^{\ast}) \widehat{x}, \widehat{x} \rangle - \langle \A(x^{\ast}) \widehat{x}, \widehat{Z} \rangle
\\
&= \langle ( \nabla_{xx}^{2} L(v^{\ast}) + {\cal H}(x^{\ast}, Z^{\ast}) ) \widehat{x}, \widehat{x} \rangle.
\end{aligned}
\end{align}
Equalities $\nabla_{x} L(v^{\ast}) = 0$ and $\nabla g(x^{\ast})^{\top} \widehat{x} = 0$ yield 
$\nabla f(x^{\ast})^{\top} \widehat{x} = \langle \nabla g(x^{\ast})^{\top} \widehat{x}, y^{\ast} \rangle + \langle \A(x^{\ast}) \widehat{x}, Z^{\ast} \rangle = {\rm tr}[P^{\top} (\A(x^{\ast}) \widehat{x}) P \Lambda] = {\rm tr}[P_{\gamma}^{\top} (\A(x^{\ast}) \widehat{x}) P_{\gamma} \Lambda_{\gamma}]$. Because the second equality of \eqref{eq:AZblock} derives ${\rm tr}[P_{\gamma}^{\top} (\A(x^{\ast}) \widehat{x}) P_{\gamma} \Lambda_{\gamma}] = 0$, it is clear that $\nabla f(x^{\ast})^{\top} \widehat{x} = 0$. Moreover, thanks to~\eqref{property:TSX} and~\eqref{eq:APZ}, it can be verified that $\A(x^{\ast}) \widehat{x} \in \T_{\S_{+}}(X(x^{\ast}))$. Thus, the definition of the critical cone $C(x^{\ast})$ implies $\widehat{x} \in C(x^{\ast})$. It then follows from (A2) and \eqref{ineq:nabLH} that
\begin{align}
\widehat{x} = 0, \label{eq:xi0}
\end{align}
and hence we readily obtain
\begin{align}
& \nabla g(x^{\ast}) \widehat{y} + \A^{\ast}(x^{\ast}) \widehat{Z} = - \nabla_{xx}^{2} L(v^{\ast}) \widehat{x} + \nabla g(x^{\ast}) \widehat{y} + \A^{\ast}(x^{\ast}) \widehat{Z} = 0,  \label{eq:gA}
\\
& {\cal P}_{\S_{+}^{d}}^{\prime}(M^{\ast}; -\widehat{Z}) = {\cal P}_{\S_{+}^{d}}^{\prime}(X(x^{\ast}) - Z^{\ast}; \A(x^{\ast}) \widehat{x} - \widehat{Z}) - \A(x^{\ast}) \widehat{x} = O.  \label{eq:VS}
\end{align}
Since (A1) is satisfied, there exist $\widehat{\eta} \in \R^{n}$ and $\widehat{\Gamma} \in \T_{\S_{+}^{d}}(X(x^{\ast})) \cap (Z^{\ast})^{\perp}$ such that
\begin{align}
\nabla g(x^{\ast})^{\top} \widehat{\eta} = -\widehat{y}, \quad \A(x^{\ast}) \widehat{\eta} + \widehat{\Gamma} = -\widehat{Z}. \label{eq:gA2}
\end{align}
Noting~\eqref{property:TSX} and $\widehat{\Gamma} \in \T_{\S_{+}^{d}}(X(x^{\ast})) \cap (Z^{\ast})^{\perp}$ yields
\begin{align*}
\left[
\begin{array}{cc}
P_{\beta}^{\top} \widehat{\Gamma} P_{\beta} & P_{\beta}^{\top} \widehat{\Gamma} P_{\gamma}
\\
P_{\gamma}^{\top} \widehat{\Gamma} P_{\beta} & P_{\gamma}^{\top} \widehat{\Gamma} P_{\gamma}
\end{array}
\right] \succeq O, \quad 0 = \langle \widehat{\Gamma}, Z^{\ast} \rangle = {\rm tr}[ P_{\gamma}^{\top} \widehat{\Gamma} P_{\gamma} \Lambda_{\gamma} ].
\end{align*}
These facts and $\Lambda_{\gamma} \succ O$ ensure $P_{\beta}^{\top} \widehat{\Gamma} P_{\beta} \succeq O$, $P_{\beta}^{\top} \widehat{\Gamma} P_{\gamma} = O$, and $P_{\gamma}^{\top} \widehat{\Gamma} P_{\gamma} = O$. On the other hand, combining \cite[Equation~(16)]{Su06} and \eqref{eq:VS} implies $P_{\alpha}^{\top} \widehat{Z} P_{\alpha} = O$, $P_{\alpha}^{\top} \widehat{Z} P_{\beta} = O$, $P_{\alpha}^{\top} \widehat{Z} P_{\gamma} = O$, and $P_{\beta}^{\top} \widehat{Z} P_{\beta} \succeq O$, where the last inequality is derived from $\P_{\S_{+}}^{|\beta|}(-P_{\beta}^{\top} \widehat{Z} P_{\beta}) = O$. It then follows from \eqref{eq:gA} and \eqref{eq:gA2} that $\Vert \widehat{y} \Vert^{2} + \Vert \widehat{Z} \Vert_{{\rm F}}^{2} = -\langle \nabla g(x^{\ast}) \widehat{y} + \A^{\ast}(x^{\ast}) \widehat{Z}, \widehat{\eta} \rangle - \langle \widehat{Z}, \widehat{\Gamma} \rangle = - {\rm tr}[ P_{\beta}^{\top} \widehat{Z} P_{\beta} P_{\beta}^{\top} \widehat{\Gamma} P_{\beta} ] 
\leq 0$.
%\begin{align*}
%\Vert \widehat{y} \Vert^{2} + \Vert \widehat{Z} \Vert_{{\rm F}}^{2}
%&= -\langle \nabla g(x^{\ast}) \widehat{y} + \A^{\ast}(x^{\ast}) \widehat{Z}, \widehat{\eta} \rangle - \langle \widehat{Z}, \widehat{\Gamma} \rangle
%\\
%&= - {\rm tr}[ P_{\beta}^{\top} \widehat{Z} P_{\beta} P_{\beta}^{\top} \widehat{\Gamma} P_{\beta} ] 
%\\
%&\leq 0.
%\end{align*}
This result and \eqref{eq:xi0} mean $(\widehat{x}, \widehat{y}, \widehat{Z}) = (0,0,O)$.
\qed
\end{proof}

\begin{theorem} \label{th:error_bound}
If~{\rm (A1)} and {\rm (A2)} hold, then there exist $r > 0$ and $c > 0$ such that $\Vert v - v^{\ast} \Vert \leq c \sigma(v)$ for all $v \in B(v^{\ast}, r)$.
\end{theorem}

\begin{proof}
We prove this theorem by contradiction. Then, we can take some sequence $\{ \widetilde{v}_{j} \coloneqq (\widetilde{x}_{j}, \widetilde{y}_{j}, \widetilde{Z}_{j}) \} \subset \V$ satisfying
\begin{align}
\textstyle \widetilde{v}_{j} \in B(v^{\ast}, \frac{1}{j}), \quad j \sigma(\widetilde{v}_{j}) < \Vert \widetilde{v}_{j} - v^{\ast} \Vert \quad \forall j \in \mathbb{N}. \label{ineq:contradiction}
\end{align}
Let $j \in \mathbb{N}$ be arbitrary and let $\tau_{j} \coloneqq \Vert \widetilde{v}_{j} - v^{\ast} \Vert$ and $\widetilde{\sigma}_{j} \coloneqq \sigma(\widetilde{v}_{j})$ for simplicity. Notice that $\tau_{j} > 0$ because the inequality of~\eqref{ineq:contradiction} does not hold when $\tau_{j} = 0$. We denote $\widehat{x}_{j} \coloneqq \frac{1}{\tau_{j}} (\widetilde{x}_{j} - x^{\ast})$, $\widehat{y}_{j} \coloneqq \frac{1}{\tau_{j}} (\widetilde{y}_{j} - y^{\ast})$, $\widehat{Z}_{j} \coloneqq \frac{1}{\tau_{j}} (\widetilde{Z}_{j} - Z^{\ast})$, and $\widehat{v}_{j} \coloneqq (\widehat{x}_{j}, \widehat{y}_{j}, \widehat{Z}_{j})$. The definition of $\tau_{j}$ implies 
\begin{align}
\max \{ \Vert \widehat{x}_{j} \Vert, \Vert \widehat{Z}_{j} \Vert_{{\rm F}} \} \leq \left\Vert \widehat{v}_{j} \right\Vert = 1. \label{ineq:xi_sigma}
\end{align}
Hence, without loss of generality, there exists $\widehat{v} \coloneqq (\widehat{x}, \widehat{y}, \widehat{Z}) \in \V$ such that
\begin{align}
\lim_{j \to \infty} \widehat{v}_{j} = \widehat{v}, \quad \widehat{v} \not= 0.  \label{noteq:hatv}
\end{align}
We take $\varepsilon > 0$ arbitrarily. From $\nabla_{x} L(v^{\ast}) = 0$, $g(x^{\ast}) = 0$, and the Taylor expansion, there exists $r_{1} > 0$ such that for all $v = (x, y, Z) \in B(v^{\ast}, r_{1})$,
\begin{align}
&\begin{aligned} \label{ineq:nabL_taylor}
& \Vert \nabla_{x} L(v) - \nabla_{xx}^{2} L(v^{\ast})(x-x^{\ast}) + \nabla g(x^{\ast})(y - y^{\ast}) + \A^{\ast}(x^{\ast})(Z - Z^{\ast}) \Vert 
\\
& \leq \frac{\varepsilon}{2} \Vert v - v^{\ast} \Vert,
\end{aligned}
\\
& \Vert g(x) - \nabla g(x^{\ast})^{\top} (x - x^{\ast}) \Vert \leq \frac{\varepsilon}{2} \Vert x - x^{\ast} \Vert, \label{ineq:eqg_taylor}
\\
& \Vert X(x) - X(x^{\ast}) - \A(x^{\ast}) (x - x^{\ast}) \Vert _{{\rm F}} \leq \frac{\varepsilon}{4} \Vert x - x^{\ast} \Vert. \label{ineq:funcX_taylor}
\end{align}
Let us define $\widehat{Q} \coloneqq \A(x^{\ast}) \widehat{x} - \widehat{Z}$. It follows from Rademacher's theorem~\cite[Section~9.J]{RoWe98} that for $\widehat{Q} \in \S^{d}$, there exists $r_{2} > 0$ such that for all $t \in (0, r_{2}]$,
\begin{align} \label{ineq:semismooth_P}
\Vert \P_{\S_{+}^{d}}(M^{\ast} + t\widehat{Q}) - \P_{\S_{+}^{d}}(M^{\ast}) - t\P_{\S_{+}^{d}}^{\prime}(M^{\ast}; \widehat{Q}) \Vert_{{\rm F}} \leq \frac{\varepsilon}{4} t. 
\end{align}
We have from \eqref{ineq:contradiction} that $\widetilde{v}_{j} \to v^{\ast}$ and $\frac{\widetilde{\sigma}_{j}}{\tau_{j}} \to 0$ as $j \to \infty$, namely, there exists $j_{1} \in \mathbb{N}$ such that
\begin{align}
\tau_{j} = \Vert \widetilde{v}_{j} - v^{\ast} \Vert \leq \min \{ r_{1}, r_{2} \}, \quad \frac{\widetilde{\sigma}_{j}}{\tau_{j}} \leq \frac{\varepsilon}{4} \quad \forall j \geq j_{1}. \label{eq:tauzero}
\end{align}
Let $j \in \mathbb{N}$ be an arbitrary integer satisfying $j \geq j_{1}$. By using \eqref{ineq:nabL_taylor}, \eqref{eq:tauzero}, and $\Vert \nabla_{x} L(\widetilde{v}_{j}) \Vert \leq \widetilde{\sigma}_{j}$, we get $\Vert \nabla_{xx}^{2} L(v^{\ast}) (\widetilde{x}_{j} - x^{\ast}) - \nabla g(x^{\ast}) (\widetilde{y}_{j} - y^{\ast}) - \A^{\ast}(x^{\ast}) (\widetilde{Z}_{j} - Z^{\ast}) \Vert \leq \frac{\varepsilon}{2} \tau_{j} + \frac{\varepsilon}{4} \tau_{j} < \varepsilon \tau_{j}$. Meanwhile, combining \eqref{ineq:eqg_taylor}, \eqref{eq:tauzero}, and $\Vert g(\widetilde{x}_{j}) \Vert \leq \widetilde{\sigma}_{j}$ implies $\Vert \nabla g(x^{\ast})^{\top}(\widetilde{x}_{j} - x^{\ast}) \Vert \leq \frac{\varepsilon}{2} \tau_{j} + \frac{\varepsilon}{4} \tau_{j} < \varepsilon \tau_{j}$. Dividing both sides of them by $\tau_{j}$ and using~\eqref{noteq:hatv} lead to
\begin{align}
\Vert \nabla_{xx}^{2} L(v^{\ast}) \widehat{x} - \nabla g(x^{\ast}) \widehat{y} - \A^{\ast}(x^{\ast}) \widehat{Z} \Vert \leq \varepsilon, \quad \Vert \nabla g(x^{\ast})^{\top} \widehat{x} \Vert \leq \varepsilon. \label{ineq:eqLgeps}
\end{align}
We now define $R_{j}$, $S_{j}$, $T_{j}$, and $U_{j}$ as follows. $R_{j} \coloneqq - X(\widetilde{x}_{j}) + X(x^{\ast}) + \A(x^{\ast}) (\widetilde{x}_{j} - x^{\ast})$, $S_{j} \coloneqq \P_{\S_{+}^{d}}(M^{\ast} + \tau_{j} \widehat{Q}) - \P_{\S_{+}^{d}}(M^{\ast}) - \tau_{j} \P_{\S_{+}^{d}}^{\prime}(M^{\ast}; \widehat{Q})$, $T_{j} \coloneqq X(\widetilde{x}_{j}) - P_{\S_{+}^{d}}(X(\widetilde{x}_{j}) - \widetilde{Z}_{j})$, and $U_{j} \coloneqq P_{\S_{+}^{d}}(X(\widetilde{x}_{j}) - \widetilde{Z}_{j}) - \P_{\S_{+}^{d}}(M^{\ast} + \tau_{j} \widehat{Q})$. Since $\A(x^{\ast}) (\widetilde{x}_{j} - x^{\ast}) - \tau_{j} \P_{\S_{+}^{d}}^{\prime}(M^{\ast}; \widehat{Q}) = \A(x^{\ast}) (\widetilde{x}_{j} - x^{\ast}) - \tau_{j} \P_{\S_{+}^{d}}^{\prime}(M^{\ast}; \widehat{Q}) + X(x^{\ast}) - \P_{\S_{+}^{d}} (M^{\ast}) = R_{j} + S_{j} + T_{j} + U_{j}$, we can easily verified that
\begin{align}
\A(x^{\ast}) \widehat{x}_{j} - \P_{\S_{+}^{d}}^{\prime}(M^{\ast}; \widehat{Q}) = \frac{1}{\tau_{j}} (R_{j} + S_{j} + T_{j} + U_{j}). \label{eq:RSTU}
\end{align}
Exploiting~\eqref{ineq:funcX_taylor}, \eqref{ineq:semismooth_P}, \eqref{eq:tauzero}, $\Vert X(\widetilde{x}_{j}) - P_{\S_{+}^{d}}(X(\widetilde{x}_{j}) - \widetilde{Z}_{j}) \Vert_{{\rm F}} \leq \widetilde{\sigma}_{j}$, and $\Vert \widetilde{x}_{j} - x^{\ast} \Vert \leq \tau_{j}$ yields that
\begin{align*}
\frac{1}{\tau_{j}} \Vert R_{j} \Vert_{{\rm F}} \leq \frac{\varepsilon}{4}, \quad \frac{1}{\tau_{j}} \Vert S_{j} \Vert_{{\rm F}} \leq \frac{\varepsilon}{4}, \quad \frac{1}{\tau_{j}} \Vert T_{j} \Vert_{{\rm F}} \leq \frac{\varepsilon}{4},
\end{align*}
and
\begin{align*}
\frac{1}{\tau_{j}} \Vert U_{j} \Vert_{{\rm F}} 
&\leq \frac{1}{\tau_{j}} \Vert X(\widetilde{x}_{j}) - X(x^{\ast}) - \widetilde{Z}_{j} + Z^{\ast} - \tau_{j} \widehat{Q} \Vert_{{\rm F}}
\\
&\leq \frac{1}{\tau_{j}} \Vert R_{j} \Vert_{{\rm F}} + \Vert \A(x^{\ast}) (\widehat{x}_{j} - \widehat{x}) \Vert_{{\rm F}} + \Vert \widehat{Z}_{j} - \widehat{Z} \Vert_{{\rm F}}
%\\
%&\leq \frac{1}{\tau_{j}} \Vert X(\widetilde{x}_{j}) - X(x^{\ast}) - \A(x^{\ast}) (\widetilde{x}_{j} - x^{\ast}) \Vert_{{\rm F}} 
%\\
%& \hspace{25mm} + \Vert \A(x^{\ast}) (\widehat{x}_{j} - \widehat{x}) \Vert_{{\rm F}} + \Vert \widehat{Z}_{j} - \widehat{Z} \Vert_{{\rm F}}
\\
&\leq \frac{\varepsilon}{4} + \Vert \A(x^{\ast}) (\widehat{x}_{j} - \widehat{x}) \Vert_{{\rm F}} + \Vert \widehat{Z}_{j} - \widehat{Z} \Vert_{{\rm F}}.
\end{align*}
These inequalities, \eqref{noteq:hatv}, and \eqref{eq:RSTU} derives
\begin{align}
\Vert \A(x^{\ast}) \widehat{x} - \P_{\S_{+}^{d}}^{\prime}(M^{\ast}; \widehat{Q}) \Vert_{{\rm F}} \leq \varepsilon. \label{ineq:AxZeps}
\end{align}
Since $\varepsilon > 0$ is arbitrary, it follows from~\eqref{ineq:eqLgeps} and \eqref{ineq:AxZeps} that $\nabla_{xx}^{2} L(v^{\ast}) \widehat{x} - \nabla g(x^{\ast}) \widehat{y} - \A^{\ast}(x^{\ast}) \widehat{Z} = 0$, $\nabla g(x^{\ast})^{\top} \widehat{x} = 0$, and $\A(x^{\ast}) \widehat{x} - \P_{\S_{+}^{d}}^{\prime}(X(x^{\ast}) - Z^{\ast}; \A(x^{\ast}) \widehat{x} - \widehat{Z}) = O$. Then, using Lemma~\ref{lemma:base} leads to $\widehat{v} = 0$. However, this contradicts to \eqref{noteq:hatv}. Therefore, the assertion is proven.
\qed
\end{proof}

The function $\sigma$ satisfies not only the error bound condition but also the locally Lipschitz continuity of $\sigma$ at $v^{\ast}$. The fact is shown below.

\begin{lemma} \label{lem:sigma_Lipschitz}
There exist $r > 0$ and $c > 0$ such that $\sigma(v) \leq c \Vert v - v^{\ast} \Vert$ for all $v \in B(v^{\ast}, r)$.
\end{lemma}

\begin{proof}
Note that $f$, $g$, and $X$ are twice continuously differentiable on $\R^{n}$. It then follows from the definition of $L$ that $\nabla_{x} L$ is locally Lipschitz continuous at $v^{\ast}$, and hence there exist $r_{1} > 0$ and $c_{1} > 0$ such that for all $v \in B(v^{\ast}, r_{1})$,
\begin{align}
\Vert \nabla_{x} L(v) - \nabla_{x} L (v^{\ast}) \Vert \leq c_{1} \Vert v - v^{\ast} \Vert. \label{ineq:LLip}
\end{align}
Moreover, since the functions $g$ and $X$ are also locally Lipschitz continuous at $x^{\ast}$, there exist $r_{2} > 0$ and $c_{2} > 0$ such that for each $x \in B(x^{\ast}, r_{2})$,
\begin{align}
\Vert g(x) - g(x^{\ast}) \Vert \leq c_{2} \Vert x - x^{\ast} \Vert, \quad \Vert X(x) - X(x^{\ast}) \Vert_{{\rm F}} \leq c_{2} \Vert x - x^{\ast} \Vert. \label{ineq:gXLip}
\end{align}
Now we define $r \coloneqq \min \{ r_{1}, r_{2} \}$ and $c: = c_{1} + 3c_{2} + 1$. By \eqref{ineq:LLip} and \eqref{ineq:gXLip}, we can easily see that for $v = (x, y, Z) \in B(v^{\ast}, r)$,
\begin{align*}
\sigma(v)
&= \Vert \nabla_{x} L(v) - \nabla_{x} L(v^{\ast}) \Vert + \Vert g(x) - g(x^{\ast}) \Vert
\\
& \hspace{10mm} + \Vert X(x) - \P_{\S_{+}^{d}}(X(x) - Z) - X(x^{\ast}) + \P_{\S_{+}^{d}}(X(x^{\ast}) - Z^{\ast}) \Vert_{{\rm F}}
\\
&\leq c_{1} \Vert v - v^{\ast} \Vert + c_{2} \Vert x - x^{\ast} \Vert
\\
& \hspace{10mm} + \Vert X(x) - X(x^{\ast}) \Vert_{{\rm F}} + \Vert \P_{\S_{+}^{d}}(X(x) - Z) - \P_{\S_{+}^{d}}(X(x^{\ast}) - Z^{\ast}) \Vert_{{\rm F}}
\\
&\leq c_{1} \Vert v - v^{\ast} \Vert + c_{2} \Vert x - x^{\ast} \Vert + 2 \Vert X(x) - X(x^{\ast}) \Vert_{{\rm F}} + \Vert Z - Z^{\ast} \Vert_{{\rm F}}
\\
&\leq c_{1} \Vert v - v^{\ast} \Vert + 3c_{2} \Vert x - x^{\ast} \Vert + \Vert Z - Z^{\ast} \Vert_{{\rm F}}
\\
&\leq c \Vert v - v^{\ast} \Vert,
\end{align*}
where the first equality is derived from $\sigma(v^{\ast}) = 0$ and the second inequality follows from the non-expansion property of $\P_{\S_{+}^{d}}$. This completes the proof.
\qed
\end{proof}

Algorithm~\ref{Local_SQSDP} repeatedly solves the following subproblem for each iteration point $v = (x, y, Z) \in \V$:
\begin{align*}
\mbox{P$(v)$} \quad
\begin{aligned}
& \mini_{(\xi, \zeta, \Sigma) \in \V} & & \langle \nabla f(x), \xi \rangle + \frac{1}{2} \langle H(v) \xi, \xi \rangle + \frac{\sigma(v)}{2} \Vert \zeta \Vert^{2} + \frac{\sigma(v)}{2} \Vert \Sigma \Vert_{{\rm F}}^{2}
\\
& \subj & & g(x) + \nabla g(x)^{\top} \xi + \sigma(v) (\zeta - y) = 0,
\\
& & & X(x) + \A(x) \xi + \sigma(v) (\Sigma - Z) \succeq O.
\end{aligned}
\end{align*}
We provide a property regarding the solvability of P$(v)$. To this end, we consider the following problem.
\begin{align*}
\mbox{P$(v^{\ast})$} \quad
\begin{aligned}
& \mini_{\xi \in \R^{n}} & & \langle \nabla f(x^{\ast}), \xi \rangle + \frac{1}{2} \langle \nabla_{xx}^{2}L(v^{\ast}) \xi, \xi \rangle
\\
& \subj & & g(x^{\ast}) + \nabla g(x^{\ast})^{\top} \xi = 0, 
\\
& & & X(x^{\ast}) + \A(x^{\ast}) \xi \succeq O.
\end{aligned}
\end{align*}
Notice that $\omega^{\ast} \coloneqq (0, y^{\ast}, Z^{\ast}) \in \V$ satisfies the KKT conditions of P$(v^{\ast})$. Moreover, if (A2) holds, then the KKT point $\omega^{\ast}$ satisfies the SOSC of P$(v^{\ast})$, i.e., there exists $\nu > 0$ such that
\begin{align}
0 < \langle \nabla f(x^{\ast}), \xi \rangle + \frac{1}{2} \langle \nabla_{xx}^{2}L(v^{\ast}) \xi, \xi \rangle \quad \forall \xi \in \Gamma \backslash \{ 0 \}, \label{ineq:SOSC_Past}
\end{align}
where $\Gamma \coloneqq \{ \xi \in \R^{n}; g(x^{\ast}) + \nabla g(x^{\ast})^{\top} \xi = 0, X(x^{\ast}) + \A(x^{\ast}) \xi \succeq O, \Vert \xi \Vert \leq \nu \}$. By using the positive constant $\nu$, we further consider the following problem.
\begin{align*}
&\mbox{Q$(v)$} \quad
\begin{aligned}
& \mini_{(\xi, \zeta, \Sigma) \in \V} & & \langle \nabla f(x), \xi \rangle + \frac{1}{2} \langle H(v) \xi, \xi \rangle + \frac{\sigma(v)}{2} \Vert \zeta \Vert^{2} + \frac{\sigma(v)}{2} \Vert \Sigma \Vert_{{\rm F}}^{2}
\\
& \subj & & g(x) + \nabla g(x)^{\top} \xi + \sigma(v) (\zeta - y) = 0,
\\
& & & X(x) + \A(x) \xi + \sigma(v) (\Sigma - Z) \succeq O,
\\
& & & \Vert \xi \Vert \leq \nu.
\end{aligned}
\end{align*}
For any iteration point $v \in \V$, problem~Q$(v)$ necessarily has a global optimum as shown below.

\begin{lemma} \label{lem:solvability}
For each $v \in \V$, there exists a global optimum of {\rm Q}$(v)$. Moreover, if {\rm (A1)} and {\rm (A2)} hold, then any global optimum $\omega(v) \coloneqq (\xi(v), \zeta(v), \Sigma(v)) \in \V$ satisfies that $\omega(v) \to \omega^{\ast}$ as $v \to v^{\ast}$.
\end{lemma}

\begin{proof}
Let ${\cal F}$ be the objective function of Q$(v)$ and let ${\cal S}$ be the feasible set of Q$(v)$. It is clear that
\begin{align} \label{ineq:Flower}
\begin{aligned}
- \nu \Vert \nabla f(x) \Vert - \nu^{2} \Vert H(v) \Vert_{{\rm F}} \leq \langle \nabla f(x), \xi \rangle + \frac{1}{2} \langle H(v) \xi, \xi \rangle \quad \forall \xi \in \nu B.
\end{aligned}
\end{align}
Using \eqref{ineq:Flower} implies $- \nu \Vert \nabla f(x) \Vert - \nu^{2} \Vert H(v) \Vert_{{\rm F}} \leq {\cal F}(v)$ for $\omega \coloneqq (\xi, \zeta, \Sigma) \in {\cal S}$. Hence, there exists $\{ \omega_{j} \coloneqq (\xi_{j}, \zeta_{j}, \Sigma_{j}) \} \subset {\cal S}$ such that ${\cal F}(\omega_{j}) \to \eta \coloneqq \inf \{ {\cal F}(\omega); \omega \in {\cal S} \} > - \infty$ as $j \to \infty$. From the boundedness of $\{ {\cal F}(\omega_{j}) \}$ and $\{ \xi_{j} \}$, it is clear that $\{ \omega_{j} \}$ is also bounded. Without loss of generality, we can assume that $\omega_{j} \to \bar{\omega}$ as $j \to \infty$. We readily have $\bar{\omega} \in {\cal S}$ because ${\cal S}$ is closed. Moreover, we obtain $| {\cal F}(\bar{\omega}) - \eta | \leq | {\cal F}(\omega_{j}) - \eta | + | {\cal F}(\omega_{j}) - {\cal F}(\omega) | \to 0$ as $j \to \infty$, that is, ${\cal F}(\bar{\omega}) = \eta$. Therefore, the existence of a global optimum is proven.
\par
From now on, we suppose that (A1) and (A2) hold. We will confirm that any global optimum $\omega(v) = (\xi(v), \zeta(v), \Sigma(v))$ satisfies
\begin{align}
\lim_{v \to v^{\ast} }\sigma(v) \zeta(v) = 0, \quad \lim_{v \to v^{\ast}} \sigma(v) \Sigma(v) = O. \label{lim:zeta_Sigma_to_yZ}
\end{align}
Now, note that $X$ is locally Lipschitz continuous at $x^{\ast}$ and $\sigma$ satisfies the error bound condition from Theorem~\ref{th:error_bound}. Hence, there exist $c_{1} > 0$ and $r_{1} > 0$ such that for all $v = (x,y,Z) \in B(v^{\ast}, r_{1})$,
\begin{align}
\Vert X(x) - X(x^{\ast}) \Vert_{{\rm F}} \leq c_{1} \sigma(v). \label{ineq:X_error_bound}
\end{align}
Let us take $v = (x, y, Z) \in B(v^{\ast}, r_{1})$ arbitrarily. Since it can be easily seen that $( 0, y - \frac{1}{\sigma(v)} g(x), Z - \frac{1}{\sigma(v)} X(x) + \frac{1}{\sigma(v)} \P_{\S_{+}^{d}}( X(x) - \sigma(v) Z ) )$ is feasible to Q$(v)$, we have
\begin{align}
\begin{aligned} \label{ineq:Qv}
& \langle \nabla f(x), \xi(v) \rangle + \frac{1}{2} \langle H(v) \xi(v), \xi(v) \rangle + \frac{\sigma(v)}{2} \Vert \zeta(v) \Vert^{2} + \frac{\sigma(v)}{2} \Vert \Sigma(v) \Vert_{{\rm F}}^{2}
\\
& ~~ \leq \frac{1}{2\sigma(v)} \Vert \sigma(v) y - g(x) \Vert^{2} 
\\
& ~~ \qquad + \frac{1}{2\sigma(v)} \Vert \sigma(v) Z - X(x) + \P_{\S_{+}^{d}}( X(x) - \sigma(v) Z) \Vert_{{\rm F}}^{2}.
\end{aligned}
\end{align}
Combining \eqref{ineq:Flower} and \eqref{ineq:Qv} implies
\begin{align} \label{ineq:sigma_zeta_Sigma}
\begin{aligned}
&\sigma(v)^{2} \Vert \zeta(v) \Vert^{2} + \sigma(v)^{2} \Vert \Sigma(v) \Vert_{{\rm F}}^{2}
\\
&~~\leq 2\sigma(v) (\nu \Vert \nabla f(x) \Vert + \nu^{2} \Vert H(v) \Vert_{{\rm F}} ) 
\\
& \hspace{10mm} + \Vert \sigma(v) y - g(x) \Vert^{2} + \Vert \sigma(v) Z - X(x) + \P_{\S_{+}^{d}}( X(x) - \sigma(v) Z) \Vert_{{\rm F}}^{2}.
\end{aligned}
\end{align}
We recall that $v \in B(v^{\ast}, r_{1})$ is arbitrary and $v^{\ast}$ satisfies that $\sigma(v^{\ast}) = 0$, $g(x^{\ast}) = 0$, and $X(x^{\ast}) = \P_{\S_{+}^{d}}(X(x^{\ast}))$. These facts and \eqref{ineq:sigma_zeta_Sigma} indicate that $\sigma(v) \zeta(v) \to 0$ and $\sigma(v) \Sigma(v) \to O$ whenever $v \to v^{\ast}$.
\par
In what follows, the latter assertion of this lemma will be proven. We begin by verifying that
\begin{align}
\lim_{v \to v^{\ast}} \xi(v) = 0. \label{lim:xiv_to_0}
\end{align}
We assume the contrary, i.e., there exist $\varepsilon_{1} \in (0, \nu)$ and $\{ \widetilde{v}_{j} \coloneqq (\widetilde{x}_{j}, \widetilde{y}_{j}, \widetilde{Z}_{j}) \} \subset \V$ such that
\begin{align}
\textstyle \widetilde{v}_{j} \in B(v^{\ast}, \frac{r_{1}}{j}), \quad \Vert \xi(\widetilde{v}_{j}) \Vert \geq \varepsilon_{1} \quad \forall j \in \mathbb{N}. \label{ineq:contradiction_xi}
\end{align}
Let us define $\widetilde{\sigma}_{j} \coloneqq \sigma(\widetilde{v}_{j})$, $\widetilde{\omega}_{j} \coloneqq \omega(\widetilde{v}_{j})$, $\widetilde{\xi}_{j} \coloneqq \xi(\widetilde{v}_{j})$, $\widetilde{\zeta}_{j} \coloneqq \zeta(\widetilde{v}_{j})$, and $\widetilde{\Sigma}_{j} \coloneqq \Sigma(\widetilde{v}_{j})$ for $j \in \mathbb{N}$. Since $\{ \widetilde{\xi}_{j} \} \subset \nu B$, we can assume without loss of generality that $\widetilde{\xi}_{j} \to \widetilde{\xi}$ as $j \to \infty$. Combining~\eqref{def:sigma}, \eqref{lim:zeta_Sigma_to_yZ}, and~\eqref{ineq:contradiction_xi} derives
\begin{gather} \label{lim:v_xi}
\begin{gathered}
\lim_{j \to \infty} \widetilde{v}_{j} = v^{\ast}, ~ \lim_{j \to \infty} \widetilde{\sigma}_{j} = 0, 
\\
\lim_{j \to \infty} \widetilde{\sigma}_{j} \widetilde{\zeta}_{j} = 0, ~ \lim_{j \to \infty} \widetilde{\sigma}_{j} \widetilde{\Sigma}_{j} = O, ~ \lim_{j \to \infty} \widetilde{\xi}_{j} = \widetilde{\xi} \not= 0. 
\end{gathered}
\end{gather}
We arbitrarily take $j \in \mathbb{N}$. Recall that $\widetilde{\omega}_{j}$ satisfies $g(\widetilde{x}_{j}) + \nabla g(\widetilde{x}_{j})^{\top} \widetilde{\xi}_{j} + \widetilde{\sigma}_{j} (\widetilde{\zeta}_{j} - \widetilde{y}_{j}) = 0$, $X(\widetilde{x}_{j}) + \A(\widetilde{x}_{j}) \widetilde{\xi}_{j} + \widetilde{\sigma}_{j} (\widetilde{\Sigma}_{j} - \widetilde{Z}_{j}) \succeq O$, and $\Vert \widetilde{\xi}_{j} \Vert \leq \nu$. It then follows from \eqref{lim:v_xi} that
\begin{align}
g(x^{\ast}) + \nabla g(x^{\ast}) \widetilde{\xi} = 0, ~ X(x^{\ast}) + \A(x^{\ast}) \widetilde{\xi} \succeq O, ~ \Vert \widetilde{\xi} \Vert \leq \nu, ~\, {\rm i.e.,} ~\, \widetilde{\xi} \in \Gamma \backslash \{ 0 \}. \label{feasible_xi}
\end{align}
By \eqref{def:sigma} and \eqref{ineq:contradiction_xi}, we get
\begin{align}
\Vert \widetilde{\sigma}_{j} \widetilde{y}_{j} - g(\widetilde{x}_{j}) \Vert \leq \widetilde{\sigma}_{j} (\Vert y^{\ast} \Vert + r_{1} + 1). \label{ineq:tilde_y}
\end{align}
Moreover, using \eqref{def:sigma}, \eqref{ineq:X_error_bound}, \eqref{ineq:contradiction_xi}, and $X(x^{\ast}) = \P_{\S_{+}^{d}}(X(x^{\ast}))$ yields
\begin{align} \label{ineq:tilde_Z}
\begin{aligned}
& \Vert \widetilde{\sigma}_{j} \widetilde{Z}_{j} - X(\widetilde{x}_{j}) + \P_{\S_{+}^{d}}( X(\widetilde{x}_{j}) - \widetilde{\sigma}_{j} \widetilde{Z}_{j}) \Vert_{{\rm F}}
\\
&~~\leq \widetilde{\sigma}_{j} \Vert \widetilde{Z}_{j} \Vert_{{\rm F}} + \Vert X(\widetilde{x}_{j}) - X(x^{\ast}) \Vert_{{\rm F}} 
\\
& \qquad \qquad + \Vert \P_{\S_{+}^{d}}(X(\widetilde{x}_{j}) - \widetilde{\sigma}_{j} \widetilde{Z}_{j}) - \P_{\S_{+}^{d}}(X(x^{\ast})) \Vert_{{\rm F}}
\\
&~~\leq 2 \widetilde{\sigma}_{j} (\Vert Z^{\ast} \Vert_{{\rm F}} + r_{1} + c_{1}). 
\end{aligned}
\end{align}
Substituting \eqref{ineq:tilde_y} and \eqref{ineq:tilde_Z} into \eqref{ineq:Qv} implies $\langle \nabla f(\widetilde{x}_{j}), \widetilde{\xi}_{j} \rangle + \frac{1}{2} \langle H(\widetilde{v}_{j}) \widetilde{\xi}_{j}, \widetilde{\xi}_{j} \rangle \leq c_{2} \widetilde{\sigma}_{j}$, where $c_{2} \coloneqq \frac{1}{2}(\Vert y^{\ast} \Vert + r_{1} + 1)^{2} + 2(\Vert Z^{\ast} \Vert_{{\rm F}} + r_{1} + c_{1})^{2}$. Then, we have by~\eqref{lim:v_xi} and~\eqref{feasible_xi} that $\langle \nabla f(x^{\ast}), \widetilde{\xi} \rangle + \frac{1}{2} \langle \nabla_{xx}^{2}L(v^{\ast}) \widetilde{\xi}, \widetilde{\xi} \rangle \leq 0$ and $\widetilde{\xi} \in \Gamma \backslash \{ 0 \}$. However, this fact contradicts to~\eqref{ineq:SOSC_Past}. Therefore, we obtain \eqref{lim:xiv_to_0}.
\par
Finally, we will prove that $\zeta(v) \to y^{\ast}$ and $\Sigma(v) \to Z^{\ast}$ as $v \to v^{\ast}$. Let us assume the contrary, namely, there exist $\varepsilon_{2} > 0$ and $\{ \widehat{v}_{j} \coloneqq (\widehat{x}_{j}, \widehat{y}_{j}, \widehat{Z}_{j}) \} \subset \V$ such that
\begin{align}
\textstyle \widehat{v}_{j} \in B(v^{\ast}, \frac{1}{j}), \quad \Vert \zeta(\widehat{v}_{j}) - y^{\ast} \Vert + \Vert \Sigma(\widehat{v}_{j}) - Z^{\ast} \Vert_{{\rm F}} \geq \varepsilon_{2} \quad \forall j \in \mathbb{N}. \label{zeta_Sigma_varepsilon}
\end{align}
We denote $\widehat{\sigma}_{j} \coloneqq \sigma(\widehat{v}_{j})$, $\widehat{\omega}_{j} \coloneqq \omega(\widehat{v}_{j})$, $\widehat{\xi}_{j} \coloneqq \xi(\widehat{v}_{j})$, $\widehat{\zeta}_{j} \coloneqq \zeta(\widehat{v}_{j})$, and $\widehat{\Sigma}_{j} \coloneqq \Sigma(\widehat{v}_{j})$ for $j \in \mathbb{N}$. By~\eqref{lim:xiv_to_0} and~\eqref{zeta_Sigma_varepsilon}, there exists $j_{1} \in \mathbb{N}$ such that
\begin{align}
\Vert \widehat{\xi}_{j} \Vert < \nu \quad \forall j \geq j_{1}. \label{ineq:nonactive}
\end{align}
Now, we recall that $\widehat{\omega}_{j}$ satisfies the KKT conditions of Q$(\widehat{v}_{j})$ because it is an optimal solution. If $j \geq j_{1}$, then \eqref{ineq:nonactive} holds, and hence $\widehat{\omega}_{j}$ satisfies
\begin{align} \label{pKKT_uv}
\begin{aligned}
& \nabla_{x} L(\widehat{x}_{j}, \widehat{\zeta}_{j}, \widehat{\Sigma}_{j}) + r_{j} = 0, 
\\
& g(\widehat{x}_{j}) + s_{j} = 0,
\\
& X(\widehat{x}_{j}) + T_{j} - \P_{\S_{+}^{d}}( X(\widehat{x}_{j}) + T_{j} - \widehat{\Sigma}_{j} ) = O,
\end{aligned}
\end{align}
where
\begin{align} \label{def:abC}
\begin{aligned}
r_{j} &\coloneqq H(\widehat{v}_{j}) \widehat{\xi}_{j}, 
\\
s_{j} &\coloneqq \nabla g(\widehat{x}_{j})^{\top} \widehat{\xi}_{j} + \widehat{\sigma}_{j}(\widehat{\zeta}_{j} - \widehat{y}_{j}),
\\
T_{j} &\coloneqq \A(\widehat{x}_{j}) \widehat{\xi}_{j} + \widehat{\sigma}_{j} (\widehat{\Sigma}_{j} - \widehat{Z}_{j}). 
\end{aligned}
\end{align}
We have from \eqref{pKKT_uv} that $(\widehat{x}_{j}, \widehat{\zeta}_{j}, \widehat{\Sigma}_{j})$ satisfies the KKT conditions of problem~\eqref{perturbedNSDP}, which is the perturbed optimization problem related to \eqref{NSDP}, with the parameter $u_{j} \coloneqq (r_{j}, s_{j}, T_{j}) \in \V$, namely,
\begin{align}
(\widehat{\zeta}_{j}, \widehat{\Sigma}_{j}) \in \U(\widehat{x}_{j}, u_{j}) \quad \forall j \geq j_{1}. \label{subset:zeta_Sgima}
\end{align}
Let us define $u_{0} \coloneqq (0,0,O) \in \V$ and $p_{0} \coloneqq (x^{\ast}, u_{0}) \in \R^{n} \times \V$. Notice that problem~\eqref{perturbedNSDP} with $u_{0}$ satisfies the RCQ at $x^{\ast}$ because (A1) holds. It then follows from \cite[Proposition~4.43]{BoSh00} that there exist $r_{2} > 0$ and $c_{2} > 0$ satisfiying
\begin{align}
\U(x,u) \subset c_{2} B \quad \forall (x,u) \in B(p_{0}, r_{2}). \label{bounded_pKKT}
\end{align}
On the other hand, \eqref{lim:zeta_Sigma_to_yZ}, \eqref{lim:xiv_to_0}, \eqref{zeta_Sigma_varepsilon}, and \eqref{def:abC} ensure that 
\begin{align}
\lim_{j \to \infty} \widehat{x}_{j} = x^{\ast}, ~ \lim_{j \to \infty} r_{j} = 0, ~ \lim_{j \to \infty} s_{j} = 0, ~ \lim_{j \to \infty} T_{j} = O, ~ ~ \lim_{j \to \infty} u_{j} = u_{0}. \label{limit:xabCu}
\end{align}
Thus, there exists $j_{2} \in \mathbb{N}$ such that $\Vert \widehat{x}_{j} - x^{\ast} \Vert + \Vert u_{j} - u_{0} \Vert \leq r_{2}$ for all $j > j_{2}$, namely,
\begin{align}
(\widehat{x}_{j}, u_{j}) \in B(p_{0}, r_{2}) \quad \forall j > j_{2}. \label{eq:uv_zero}
\end{align}
Exploiting \eqref{subset:zeta_Sgima}, \eqref{bounded_pKKT}, and \eqref{eq:uv_zero} derives $(\widehat{\zeta}_{j}, \widehat{\Sigma}_{j}) \in c_{2} B$ for every $j > j_{3}$, where $j_{3} \coloneqq \max \{ j_{1}, j_{2} \}$. Then, we also have $(\widehat{\zeta}_{j}, \widehat{\Sigma}_{j}) \in c_{3} B$ for each $j \in \mathbb{N}$, where $c_{3} \coloneqq \max \{ c_{2}, \Vert \widehat{\zeta}_{1} \Vert + \Vert \widehat{\Sigma}_{1} \Vert_{{\rm F}}, \ldots, \Vert \widehat{\zeta}_{j_{3}} \Vert + \Vert \widehat{\Sigma}_{j_{3}} \Vert_{{\rm F}} \}$. Since $\{ (\widehat{\zeta}_{j}, \widehat{\Sigma}_{j}) \}$ is bounded, we may assume without loss of generality that
\begin{align}
\lim_{j \to \infty} \widehat{\zeta}_{j} = \widehat{\zeta}, \quad \lim_{j \to \infty} \widehat{\Sigma}_{j} = \widehat{\Sigma}. \label{lim:zeta_Sigma}
\end{align}
Combining~\eqref{pKKT_uv}, \eqref{limit:xabCu}, and \eqref{lim:zeta_Sigma} implies $\nabla_{x} L(x^{\ast}, \widehat{\zeta}, \widehat{\Sigma}) = 0$, $g(x^{\ast}) = 0$, and $X(x^{\ast}) - P_{\S_{+}^{d}}( X(x^{\ast}) - \widehat{\Sigma} ) = O$. It then follows from $\M(x^{\ast}) = \{ (y^{\ast}, Z^{\ast}) \}$ that $\widehat{\zeta} = y^{\ast}$ and $\widehat{\Sigma} = Z^{\ast}$. However, they lead to a contradiction because $\Vert \widehat{\zeta} - y^{\ast} \Vert + \Vert \widehat{\Sigma} - Z^{\ast} \Vert_{{\rm F}} \geq \varepsilon_{2}$ by~\eqref{zeta_Sigma_varepsilon} and \eqref{lim:zeta_Sigma}. Therefore, we obtain $\omega(v) \to \omega^{\ast}$ as $v \to v^{\ast}$. 
\qed
\end{proof}

Thanks to Lemma~\ref{lem:solvability}, it is verified that $\xi(v) \to 0$ as $v \to v^{\ast}$. That is to say, the constraint $\Vert \xi \Vert \leq \nu$ is redundant for $v$ sufficiently close to $v^{\ast}$. By exploiting this property, we will ensure the solvability of P$(v)$.

\begin{proposition} \label{pro:solvable}
Suppose that {\rm (A1)} and {\rm (A2)} are satisfied. For any $v \in \V$, let $\omega(v) \coloneqq (\xi(v), \zeta(v), \Sigma(v))$ be a global optimum of {\rm Q}$(v)$. Then, there exist $r > 0$ and $c > 0$ such that $\omega(v)$ is a local optimum of {\rm P}$(v)$ for all $v \in B(v^{\ast}, r)$ and satisfies that
\begin{gather*}
\Vert \delta(v) \Vert \leq c \sigma(v), \quad \Vert v + \delta (v) - v^{\ast} \Vert \leq c \sigma(v),
\\
\Vert \xi(v) \Vert \leq c \sigma(v), \quad \Vert \zeta(v) - y \Vert \leq c \sigma(v), \quad \Vert \Sigma(v) - Z \Vert_{{\rm F}} \leq c \sigma(v),
\end{gather*}
where $v \coloneqq (x, y, Z)$ and $\delta (v) \coloneqq (\xi(v), \zeta(v) - y, \Sigma(v) - Z)$.
\end{proposition}

\begin{proof}
Lemma~\ref{lem:solvability} guarantees that Q$(v)$ has a global optimum for all $v \in \V$. Moreover, since (A1) and (A2) are satisfied, any global optimum $\omega(v) = (\xi(v), \zeta(v), \Sigma(v)) \in \V$ satisfies that
\begin{align}
\lim_{v \to v^{\ast}} \xi(v) = 0. \label{lim:xito0}
\end{align}
Hence, there exists $r_{1} > 0$ such that $\Vert \xi(v) \Vert < \nu$ for any $v \in B(v^{\ast}, r_{1})$. This fact implies that $\omega(v)$ is a local optimum of P$(v)$ whenever $v \in B(v^{\ast}, r_{1})$.
\par
The locally Lipschitz continuity of $X$ around $x^{\ast}$ and Theorem~\ref{th:error_bound} guarantee that there exist $r_{2} > 0$ and $c_{1} > 0$ such that for $v = (x,y,Z) \in B(v^{\ast}, r_{2})$,
\begin{align}
\Vert X(x) - X(x^{\ast}) \Vert_{{\rm F}} \leq c_{1} \sigma(v), \quad \Vert v - v^{\ast} \Vert \leq c_{1} \sigma(v). \label{ineq:error_bound_sigma}
\end{align}
We begin by proving that there exist $r_{3} \in (0, r_{1})$ and $c_{2} > 0$ such that for $v \in B(v^{\ast}, r_{3})$,
\begin{align}
\Vert \omega(v) - \omega^{\ast} \Vert \leq c_{2} \sigma(v). \label{ineq:omega_sigma}
\end{align}
To this end, we assume the contrary, that is to say, there exists $\{ \widetilde{v}_{j} \coloneqq (\widetilde{x}_{j}, \widetilde{y}_{j}, \widetilde{Z}_{j}) \} \subset \V$ such that
\begin{align}
\textstyle \widetilde{v}_{j} \in B(v^{\ast}, \frac{r_{1}}{j}), \quad j \sigma(\widetilde{v}_{j}) < \Vert \omega(\widetilde{v}_{j}) - \omega^{\ast} \Vert \quad \forall j \in \mathbb{N}. \label{sigma_omega_varepsilon}
\end{align}
Now, we denote $\widetilde{\sigma}_{j} \coloneqq \sigma(\widetilde{v}_{j})$, $\widetilde{\omega}_{j} \coloneqq \omega(\widetilde{v}_{j})$, $\widetilde{\xi}_{j} \coloneqq \xi(\widetilde{v}_{j})$, $\widetilde{\zeta}_{j} \coloneqq \zeta(\widetilde{v}_{j})$, $\widetilde{\Sigma}_{j} \coloneqq \Sigma(\widetilde{v}_{j})$, and $\tau_{j} \coloneqq \Vert \widetilde{\omega}_{j} - \omega^{\ast} \Vert$ for $j \in \mathbb{N}$. It can be easily verified that $\tau_{j} > 0$ for every $j \in \mathbb{N}$ because~\eqref{sigma_omega_varepsilon} does not hold when $\tau_{j} = 0$. Moreover, let us define $\widehat{\xi}_{j} \coloneqq \frac{1}{\tau_{j}} \widetilde{\xi}_{j}$, $\widehat{\zeta}_{j} \coloneqq \frac{1}{\tau_{j}} (\widetilde{\zeta}_{j} - y^{\ast})$, $\widehat{\Sigma}_{j} \coloneqq \frac{1}{\tau_{j}} (\widetilde{\Sigma}_{j} - Z^{\ast})$, and $\widehat{\omega}_{j} \coloneqq (\widehat{\xi}_{j}, \widehat{\zeta}_{j}, \widehat{\Sigma}_{j})$ for $j \in \mathbb{N}$. From these definitions, it is clear that $\Vert \widehat{\omega}_{j} \Vert = 1$ for all $j \in \mathbb{N}$. Thus, we can assume without loss of generality that there exists $\widehat{\omega} \coloneqq (\widehat{\xi}, \widehat{\zeta}, \widehat{\Sigma}) \in \V$ such that
\begin{align}
\lim_{j \to \infty} \widehat{\omega}_{j} = \widehat{\omega}, \quad \widehat{\omega} \not = 0. \label{lim:omega_not_zero}
\end{align}
Lemma~\ref{lem:solvability}, \eqref{def:sigma}, and \eqref{sigma_omega_varepsilon} indicate
\begin{align}
\lim_{j \to \infty} \Vert \widetilde{v}_{j} - v^{\ast} \Vert = 0, \quad \lim_{j \to \infty} \tau_{j} = 0, \quad \lim_{j \to \infty} \widetilde{\sigma}_{j} = 0, \quad  \lim_{j \to \infty} \frac{\widetilde{\sigma}_{j}}{\tau_{j}} = 0. \label{lim:4kinds}
\end{align}
Moreover, since $\widetilde{v}_{j} \to v^{\ast}$ as $j \to \infty$, there exists $j_{1} \in \mathbb{N}$ such that
\begin{align}
\widetilde{v}_{j} \in B(v^{\ast}, r_{2}) \quad \forall j \geq j_{1}. \label{vi_in_Br3}
\end{align}
We arbitrarily take $j \geq j_{1}$. The first part of this proof and \eqref{sigma_omega_varepsilon} ensure that $\widetilde{\omega}_{j}$ satisfies the KKT conditions of P$(\widetilde{v}_{j})$, and hence
\begin{align}
& H(\widetilde{v}_{j}) \widetilde{\xi}_{j} + \nabla f(\widetilde{x}_{j}) - \nabla g(\widetilde{x}_{j}) \widetilde{\zeta}_{j} - \A^{\ast}(\widetilde{x}_{j}) \widetilde{\Sigma}_{j} = 0, \label{eq:subKKT1}
\\
& g(\widetilde{x}_{j}) + \nabla g(\widetilde{x}_{j})^{\top} \widetilde{\xi}_{j} + \widetilde{\sigma}_{j} (\widetilde{\zeta}_{j} - \widetilde{y}_{j}) = 0, \label{eq:subKKT2}
\\
&\begin{aligned} \label{eq:subKKT3}
& X(\widetilde{x}_{j}) + \A(\widetilde{x}_{j}) \widetilde{\xi}_{j} + \widetilde{\sigma}_{j} (\widetilde{\Sigma}_{j} - \widetilde{Z}_{j})
\\
& \hspace{15mm} = \P_{\S_{+}^{d}}(X(\widetilde{x}_{j}) + \A(\widetilde{x}_{j}) \widetilde{\xi}_{j} + \widetilde{\sigma}_{j} (\widetilde{\Sigma}_{j} - \widetilde{Z}_{j}) - \widetilde{\Sigma}_{j}).
\end{aligned}
\end{align}
Using \eqref{eq:subKKT1} implies $H(\widetilde{v}_{j}) \widetilde{\xi}_{j} - \nabla g(\widetilde{x}_{j}) ( \widetilde{\zeta}_{j} - y^{\ast} ) - \A^{\ast}(\widetilde{x}_{j}) ( \widetilde{\Sigma}_{j} - Z^{\ast} ) = - \nabla_{x}L(\widetilde{v}_{j}) - \nabla g(\widetilde{x}_{j}) (\widetilde{y}_{j} - y^{\ast}) - \A^{\ast}(\widetilde{x}_{j}) (\widetilde{Z}_{j} - Z^{\ast})$. Dividing both sides of the equality by $\tau_{j}$ and combining~\eqref{ineq:error_bound_sigma}, \eqref{vi_in_Br3}, and $\Vert \nabla_{x} L(\widetilde{x}_{j}) \Vert \leq \widetilde{\sigma}_{j}$ derive
\begin{align}
\begin{aligned} \label{ineq:HgA}
& \Vert H(\widetilde{v}_{j}) \widehat{\xi}_{j} - \nabla g(\widetilde{x}_{j}) \widehat{\zeta}_{j} - \A^{\ast}(\widetilde{x}_{j}) \widehat{\Sigma}_{j} \Vert
\\
& ~~ \leq \frac{1}{\tau_{j}} \Vert \nabla_{x} L(\widetilde{v}_{j}) \Vert + \frac{1}{\tau_{j}} \left( \Vert \nabla g(\widetilde{x}_{j}) \Vert + \sum_{\ell = 1}^{n} \Vert A_{\ell}(\widetilde{x}_{j}) \Vert_{2} \right) \Vert \widetilde{v}_{j} - v^{\ast} \Vert
\\
& ~~ \leq \frac{\widetilde{\sigma}_{j}}{\tau_{j}} \left( 1 + c_{1} \Vert \nabla g(\widetilde{x}_{j}) \Vert + c_{1} \sum_{\ell = 1}^{n} \Vert A_{\ell}(\widetilde{x}_{j}) \Vert_{2} \right).
\end{aligned}
\end{align}
We have from \eqref{eq:subKKT2} that $\nabla g(\widetilde{x}_{j})^{\top} \widetilde{\xi}_{j} = - g(\widetilde{x}_{j}) - \widetilde{\sigma}_{j} ( \widetilde{\zeta}_{j} - y^{\ast} ) + \widetilde{\sigma}_{j} (\widetilde{y}_{j} - y^{\ast})$. Dividing both sides of the equality by $\tau_{j}$ and exploiting $\Vert g(\widetilde{x}_{j}) \Vert \leq \widetilde{\sigma}_{j}$ yield
\begin{align} \label{ineq:sub_eq_g}
\Vert \nabla g(\widetilde{x}_{j})^{\top} \widehat{\xi}_{j} \Vert \leq \frac{\widetilde{\sigma}_{j}}{\tau_{j}} (1 + \Vert \widetilde{y}_{j} - y^{\ast} \Vert) + \widetilde{\sigma}_{j} \Vert \widehat{\zeta}_{j} \Vert.
\end{align}
Combining \eqref{lim:omega_not_zero}, \eqref{lim:4kinds}, \eqref{ineq:HgA}, and \eqref{ineq:sub_eq_g} yields
\begin{align}
\nabla_{xx}^{2} L(v^{\ast}) \widehat{\xi} - \nabla g(x^{\ast}) \widehat{\zeta} - \A^{\ast}(x^{\ast}) \widehat{\Sigma} = 0, \quad \nabla g(x^{\ast})^{\top} \widehat{\xi} = 0. \label{eq:basic_lemma_eq12}
\end{align}
Let us define
\begin{align*}
\widehat{Q} & \coloneqq \A(x^{\ast}) \widehat{\xi} - \widehat{\Sigma},
\\
R_{j} & \coloneqq X(x^{\ast}) - X(\widetilde{x}_{j}) - \widetilde{\sigma}_{j} (\widetilde{\Sigma}_{j} - \widetilde{Z}_{j}),
\\
S_{j} & \coloneqq \P_{\S_{+}^{d}}(X(\widetilde{x}_{j}) + \A(\widetilde{x}_{j}) \widetilde{\xi}_{j} + \widetilde{\sigma}_{j} (\widetilde{\Sigma}_{j} - \widetilde{Z}_{j}) - \widetilde{\Sigma}_{j}) - \P_{\S_{+}^{d}} (M^{\ast} + \tau_{j} \widehat{Q}),
\\
T_{j} & \coloneqq \P_{\S_{+}^{d}} (M^{\ast} + \tau_{j} \widehat{Q}) - \P_{\S_{+}^{d}}(M^{\ast}) - \tau_{j} \P_{\S_{+}^{d}}^{\prime}(M^{\ast}; \widehat{Q}).
\end{align*}
By these definitions and \eqref{eq:subKKT3}, we obtain $\A(\widetilde{x}_{j}) \widetilde{\xi}_{j} - \tau_{j} \P_{\S_{+}^{d}}^{\prime} (M^{\ast}; \widehat{Q}) = R_{j} + S_{j} + T_{j}$. Dividing both sides of the equality by $\tau_{j}$ yields
\begin{align}
\Vert \A(\widetilde{x}_{j}) \widehat{\xi}_{j} - \P_{\S_{+}^{d}}^{\prime} (M^{\ast}; \widehat{Q}) \Vert_{{\rm F}} \leq \frac{1}{\tau_{j}} (\Vert R_{j} \Vert_{{\rm F}} + \Vert S_{j} \Vert_{{\rm F}} + \Vert T_{j} \Vert_{{\rm F}}). \label{ineq:matRST}
\end{align}
Meanwhile, using \eqref{ineq:error_bound_sigma} and the non-expansion property of $\P_{\S_{+}^{d}}$ derives
\begin{align*}
& \frac{1}{\tau_{j}} \Vert R_{j} \Vert_{{\rm F}} \leq \frac{\widetilde{\sigma}_{j}}{\tau_{j}} ( c_{1} + \Vert \widetilde{Z}_{j} - Z^{\ast} \Vert_{{\rm F}} ) + \widetilde{\sigma}_{j} \Vert \widehat{\Sigma}_{j} \Vert_{{\rm F}}, 
\\
& \frac{1}{\tau_{j}} \Vert S_{j} \Vert_{{\rm F}} \leq \frac{1}{\tau_{j}} \Vert R_{j}  \Vert_{{\rm F}} + \Vert \A(\widetilde{x}_{j}) \widehat{\xi}_{j} - \A(x^{\ast}) \widehat{\xi} \Vert_{{\rm F}} + \Vert \widehat{\Sigma}_{j} - \widehat{\Sigma} \Vert_{{\rm F}}.
\end{align*}
It then follows from \eqref{lim:omega_not_zero} and \eqref{lim:4kinds} that $\frac{1}{\tau_{j}} \Vert R_{j} \Vert_{{\rm F}} \to 0$ and $\frac{1}{\tau_{j}} \Vert S_{j} \Vert_{{\rm F}} \to 0$ as $j \to \infty$. Moreover, the directional differentiability of $\P_{\S_{+}^{d}}$ and \eqref{lim:4kinds} ensure $\frac{1}{\tau_{j}} \Vert T_{j} \Vert_{{\rm F}} \to 0$. These facts and \eqref{ineq:matRST} imply
\begin{align}
\A(x^{\ast}) \widehat{\xi} - \P_{\S_{+}^{d}}^{\prime} (X(x^{\ast}) - Z^{\ast}; \A(x^{\ast}) \widehat{\xi} - \widehat{\Sigma}) = O. \label{eq:basic_lemma_eq3}
\end{align}
Since \eqref{eq:basic_lemma_eq12} and \eqref{eq:basic_lemma_eq3} hold, Lemma~\ref{lemma:base} can be applied. As a result, we get $\widehat{\omega} = 0$. However, this contradicts to \eqref{lim:omega_not_zero}. Therefore, it can be verified that \eqref{ineq:omega_sigma} holds. Recall that $\sigma(v) \to 0$ as $v \to v^{\ast}$. Then, we have by~\eqref{ineq:omega_sigma} that there exists $r_{4} > 0$ such that
\begin{align}
\Vert \omega(v) - \omega^{\ast} \Vert \leq \nu \quad \forall v \in B(v^{\ast}, r_{4}). \label{ineq:omega_r0}
\end{align}
Let us denote $r \coloneqq \min \{ r_{2}, r_{3}, r_{4} \}$ and $c \coloneqq 2c_{1} + c_{2}$. We take $v \in B(v^{\ast}, r)$ arbitrarily. Since $r \leq r_{3} < r_{1}$ holds, the first part of the proof guarantees that $\omega(v)$ is a KKT point of P$(v)$. Moreover, it is clear that $\omega(v) \in B(\omega^{\ast}, \nu)$ from $r \leq r_{4}$ and \eqref{ineq:omega_r0}. Combining~\eqref{ineq:error_bound_sigma} and \eqref{ineq:omega_sigma} yields
\begin{align}
\begin{aligned} \label{ineq:delta_v_c}
\Vert \delta(v) \Vert \leq \Vert \omega(v) - \omega^{\ast} \Vert + \Vert y - y^{\ast} \Vert + \Vert Z - Z^{\ast} \Vert_{{\rm F}} \leq (c_{1} + c_{2}) \sigma(v),
\end{aligned}
\end{align}
Noting the definitions of $\delta(v)$ and $c$ implies that $\Vert \delta(v) \Vert \leq c\sigma(v)$, $\Vert \xi(v) \Vert \leq c \sigma(v)$, $\Vert \zeta(v) - y \Vert \leq c \sigma(v)$, and $\Vert \Sigma(v) - Z \Vert_{{\rm F}} \leq c \sigma(v)$. We have from~\eqref{ineq:error_bound_sigma} and~\eqref{ineq:delta_v_c} that $\Vert v + \delta(v) - v^{\ast} \Vert \leq \Vert v - v^{\ast} \Vert + \Vert \delta(v) \Vert \leq c \sigma(v)$.
\qed
\end{proof}

Now, we provide several properties that directly affect the convergence rate of Algorithm~\ref{Local_SQSDP}.

\begin{lemma} \label{lem:sigma_property}
Suppose that {\rm (A1)}, {\rm (A2)}, and {\rm (A3)} are satisfied. Then, there exists $r > 0$ such that for any $v \in B(v^{\ast}, r)$, problem {\rm P}$(v)$ has a local optimum $\omega(v) \coloneqq (\xi(v), \zeta(v), \Sigma(v)) \in B(\omega^{\ast}, r)$, and $\omega(v)$ satisfies the following statements:
\begin{description}
\item[{\rm (a)}] For any $\epsilon > 0$, there exists $\eta > 0$ such that $\sigma(v + \delta(v)) \leq \epsilon \sigma(v)$ for any $v \in B(v^{\ast}, \eta)$, where $v \coloneqq (x, y, Z)$ and $\delta(v) \coloneqq (\xi(v), \zeta(v)-y, \Sigma(v) - Z)$;

\item[{\rm (b)}] if $H(v) - \nabla_{xx}^{2}L(v^{\ast}) = {\cal O}(\Vert v- v^{\ast} \Vert)) ~ (v \to v^{\ast})$, then there exist $\mu > 0$ and $c > 0$ such that $\sigma(v + \delta(v)) \leq c \sigma(v)^{2}$ for all $v \in B(v^{\ast}, \mu)$.
\end{description}
\end{lemma}

\begin{proof}
From Theorem~\ref{th:error_bound} and Proposition~\ref{pro:solvable}, there exist $r_{1} > 0$ and $c_{1} > 0$ such that for each $v = (x, y, Z) \in B(v^{\ast}, r_{1})$, problem {\rm P}$(v)$ has a local optimum $\omega(v) = (\xi(v), \zeta(v), \Sigma(v)) \in B(\omega^{\ast}, \nu)$ satisfying
\begin{align} \label{ineq:delta_v}
\begin{aligned}
& \Vert v - v^{\ast} \Vert \leq c_{1} \sigma(v), & & \Vert \delta(v) \Vert \leq c_{1} \sigma(v), & & \Vert v + \delta(v) - v^{\ast} \Vert \leq c_{1} \sigma(v), 
\\
& \Vert \xi(v) \Vert \leq c_{1} \sigma(v), & & \Vert \zeta(v) - y \Vert \leq c_{1} \sigma(v), & & \Vert \Sigma(v) - Z \Vert_{{\rm F}} \leq c_{1} \sigma(v),
\end{aligned}
\end{align} 
where $\delta(v) = (\xi(v), \zeta(v) - y, \Sigma(v) - Z)$. Since $\omega(v)$ satisfies the KKT conditions of {\rm P}$(v)$ for $v = (x, y, Z) \in B(v^{\ast}, r_{1})$, we have
\begin{align}
& H(v) \xi(v) + \nabla f(x) - \nabla g(x) \zeta(v) - \A^{\ast}(x) \Sigma(v) = 0, \label{subKKT1}
\\
& g(x) + \nabla g(x)^{\top} \xi(v) + \sigma(v) (\zeta(v) - y) = 0, \label{subKKT2}
\\
&
\begin{aligned} \label{subKKT3}
& X(x) + \A(x) \xi(v) + \sigma(v) (\Sigma(v) - Z) 
\\
& \hspace{15mm} = \P_{\S_{+}^{d}}(X(x) + \A(x) \xi(v) + \sigma(v) (\Sigma(v) - Z) - \Sigma(v)). 
\end{aligned}
\end{align}
Note that $g$ and $X$ are twice continuously differentiable on $\R^{n}$, $\nabla_{x} L$ is continuously differentiable on $\V$, and $\nabla_{xx}^{2} L$ is locally Lipschitz continuous at $v^{\ast}$, where the locally Lipschitz continuity of $\nabla_{xx}^{2} L$ is obtained from (A3). These facts guarantee that there exist $r_{2} > 0$ and $c_{2} > 0$ such that
\begin{align}
&\hspace{-2.6mm}
\begin{aligned} \label{ineq:nab_Lag}
& \Vert \nabla_{x} L(\widetilde{v}) - \nabla_{x} L(\widehat{v}) - \nabla_{xx}^{2} L(\widehat{v}) (\widetilde{x} - \widehat{x}) + \nabla g(\widehat{x}) (\widetilde{y} - \widehat{y}) + \A^{\ast}(\widehat{x}) (\widetilde{Z} - \widehat{Z}) \Vert 
\\
& \leq c_{2} \Vert \widetilde{v} - \widehat{v} \Vert^{2}, 
\end{aligned}
\\
&\hspace{-2mm} \Vert \nabla_{xx}^{2} L(\widetilde{v}) - \nabla_{xx}^{2} L(\widehat{v}) \Vert_{2} \leq c_{2} \Vert \widetilde{v} - \widehat{v} \Vert, \label{ineq:nab_Lag2}
\\
&\hspace{-2mm} \Vert g(\widetilde{x}) - g(\widehat{x}) - \nabla g(\widehat{x})^{\top} (\widetilde{x} - \widehat{x}) \Vert \leq c_{2} \Vert \widetilde{x} - \widehat{x} \Vert^{2}, \label{ineq:eq_g}
\\
&\hspace{-2mm} \Vert X(\widetilde{x}) - X(\widehat{x}) - \A(\widehat{x}) (\widetilde{x} - \widehat{x}) \Vert_{{\rm F}} \leq c_{2} \Vert \widetilde{x} - \widehat{x} \Vert^{2} \label{ineq:func_X}
\end{align}
for $\widetilde{v} \coloneqq (\widetilde{x}, \widetilde{y}, \widetilde{Z}) \in B(v^{\ast}, r_{2})$ and $\widehat{v} \coloneqq (\widehat{x}, \widehat{y}, \widehat{Z}) \in B(v^{\ast}, r_{2})$. The continuity of $\sigma$ and \eqref{ineq:delta_v} ensure the existence of $r_{3} >0$ such that
\begin{align}
\Vert v + \delta(v) - v^{\ast} \Vert \leq c_{1} \sigma(v) \leq \min \{ r_{1}, r_{2} \} \label{ineq:delta_v2}
\end{align}
for any $v \in B(v^{\ast}, r_{3})$. Let $r_{4} \coloneqq \min \{ r_{1}, r_{2}, r_{3} \}$. We have from \eqref{subKKT1} that
\begin{align*}
& \nabla_{x} L(v + \delta(v)) 
\\
& = \nabla_{x} L(v + \delta(v)) - H(v) \xi(v) - \nabla f(x) + \nabla g(x) \zeta(v) + \A^{\ast}(x) \Sigma(v)
\\
& = \nabla_{x} L(v + \delta(v)) - \nabla_{x} L(v) - \nabla_{xx}^{2} L(v) \xi(v)
\\
& \hspace{10mm} + \nabla g(x) (\zeta(v)-y)  + \A^{\ast}(x) (\Sigma(v)-Z) + \nabla_{xx}^{2} L(v) \xi(v) - H(v)\xi(v) ,
\end{align*}
and hence combining \eqref{ineq:delta_v}, \eqref{ineq:delta_v}, \eqref{ineq:nab_Lag}, \eqref{ineq:nab_Lag2}, and \eqref{ineq:delta_v2} yields
\begin{align}
\begin{aligned}\label{ineq:nabL_next1}
\Vert \nabla_{x} L(v + \delta(v)) \Vert 
&\leq c_{2} \Vert \delta(v) \Vert^{2} + \Vert H(v) - \nabla_{xx}^{2} L(v) \Vert_{2} \Vert \xi(v) \Vert 
\\
&\leq c_{1}^{2} c_{2} \sigma(v)^{2} + c_{1} \sigma(v) \Vert H(v) - \nabla_{xx}^{2} L(v^{\ast}) \Vert_{2} 
\\
& \hspace{25mm} + c_{1} \sigma(v) \Vert \nabla_{xx}^{2} L(v^{\ast}) - \nabla_{xx}^{2} L(v) \Vert_{2} 
\\
&\leq 2c_{1}^{2} c_{2} \sigma(v)^{2} + c_{1} \sigma(v) \Vert H(v) - \nabla_{xx}^{2} L(v^{\ast}) \Vert_{2} 
\end{aligned}
\end{align}
for every $v \in B(v^{\ast}, r_{4})$. On the other hand, exploiting \eqref{subKKT2} and \eqref{ineq:delta_v} implies that $\Vert g(x) + \nabla g(x)^{\top} \xi(v)\Vert = \sigma(v) \Vert (\zeta(v) - y) \Vert \leq c_{1} \sigma(v)^{2}$. From this inequality, \eqref{ineq:delta_v}, \eqref{ineq:eq_g}, and \eqref{ineq:delta_v2}, it can be verified that for any $v \in B(v^{\ast}, r_{4})$,
\begin{align}
\Vert g(x + \xi(v)) \Vert \leq ( c_{1}^{2} c_{2} + c_{1} ) \sigma(v)^{2}. \label{ineq:eqg_next2}
\end{align}
For simplicity, we use the following notation:
\begin{align*}
R(v) &\coloneqq X(x + \xi(v)) - X(x) - \A(x) \xi(v),
\\
S(v) &\coloneqq X(x) + \A(x) \xi(v) + \sigma(v)(\Sigma(v) - Z) - \Sigma(v),
\\
T(v) &\coloneqq X(x) + \A(x) \xi(v) + R(v) - \Sigma(v).
\end{align*}
By using \eqref{ineq:func_X}, \eqref{ineq:delta_v2}, and \eqref{ineq:delta_v}, we can evaluate $R(v)$ as follows: For $v \in B(v^{\ast}, r_{4})$,
\begin{align}
\Vert R(v) \Vert_{{\rm F}} \leq c_{2} \Vert \xi(v) \Vert^{2} \leq c_{1}^{2} c_{2} \sigma(v)^{2}. \label{ineq:mat_R}
\end{align}
It follows from \eqref{subKKT3}, \eqref{ineq:mat_R}, \eqref{ineq:delta_v}, and the non-expansion property of $\P_{\S_{+}^{d}}$ that
\begin{align}
\hspace{-2.5mm}
\begin{aligned} \label{ineq:funcX_next3}
& \Vert X(x + \xi(v)) - \P_{\S_{+}^{d}}(X(x + \xi(v)) - \Sigma(v)) \Vert_{{\rm F}} 
\\
& = \Vert X(x) + \A(x) \xi(v) + R(v) - \P_{\S_{+}^{d}}(X(x) + \A(x) \xi(v) + R(v) - \Sigma(v)) \Vert_{{\rm F}} 
\\
& \leq \Vert \P_{\S_{+}^{d}}(S(v)) - \P_{\S_{+}^{d}}(T(v)) \Vert_{{\rm F}} + \Vert R(v) \Vert_{{\rm F}} + \sigma(v) \Vert \Sigma(v) - Z \Vert_{{\rm F}} 
\\
& \leq \Vert S(v) - T(v) \Vert_{{\rm F}} + \Vert R(v) \Vert_{{\rm F}} + \sigma(v) \Vert \Sigma(v) - Z \Vert_{{\rm F}} 
\\
& \leq 2 \Vert R(v) \Vert_{{\rm F}} + 2 \sigma(v) \Vert \Sigma(v) - Z \Vert_{{\rm F}} 
\\
& \leq 2(c_{1}^{2} c_{2} + c_{1}) \sigma(v)^{2} 
\end{aligned}
\end{align}
for all $v =(x, y, Z) \in B(v^{\ast}, r_{4})$. In the following, we show the assertions of the lemma by utilizing the above results.
\par
We prove item (a). Let $\epsilon > 0$ be arbitrary. Since $H(v) \to \nabla_{xx}^{2} L(v^{\ast})$ and $\sigma(v) \to 0$ as $v \to v^{\ast}$, there exists $r_{5} > 0$ such that for each $v \in B(v^{\ast}, r_{5})$,
\begin{align}
\Vert H(v) - \nabla_{xx}^{2}L(v^{\ast}) \Vert_{2} \leq \frac{\epsilon}{2c_{1}}, \quad \sigma(v) \leq \frac{\epsilon}{2(5 c_{1}^{2} c_{2} + 3 c_{1})}. \label{ineq:H_sigma}
\end{align}
Let $\eta \coloneqq \min \{ r_{4}, r_{5} \} > 0$. Combining~\eqref{def:sigma}, \eqref{ineq:delta_v}, \eqref{ineq:nabL_next1}, \eqref{ineq:eqg_next2}, and \eqref{ineq:funcX_next3} implies $\sigma( v + \delta(v) ) \leq \epsilon \sigma(v)$ for all $v \in B(v^{\ast}, \eta)$. Item~(a) is proven.
\par
Now suppose that the assumption of item~(b) holds, i.e., there exist $r_{6} > 0$ and $c_{3} > 0$ such that $\Vert H(v) - \nabla_{xx}^{2}L(v^{\ast}) \Vert_{2} \leq c_{3} \Vert v - v^{\ast} \Vert$ for $v \in B(v^{\ast}, r_{6})$. If we define $\mu \coloneqq \min \{ r_{4}, r_{6} \} > 0$ and $c \coloneqq 5c_{1}^{2} c_{2} + c_{1}^{2} c_{3} + 3 c_{1} > 0$, then~\eqref{def:sigma}, \eqref{ineq:delta_v}, \eqref{ineq:nabL_next1}, \eqref{ineq:eqg_next2}, and \eqref{ineq:funcX_next3} derive $\sigma( v + \delta(v) ) \leq c \sigma(v)^{2}$ for $v \in B(v^{\ast}, \mu)$. The result shows that item~(b) holds.
\qed
\end{proof}

\begin{theorem}
Suppose that {\rm (A1)}, {\rm (A2)}, and {\rm (A3)} hold. Then, there exists $r > 0$ such that if the initial point of Algorithm~{\rm \ref{Local_SQSDP}}, say $v_{0} \in \V$, satisfies $v_{0} \in B(v^{\ast}, r)$, Algorithm~{\rm \ref{Local_SQSDP}} is well defined and the generated sequence $\{ v_{k} \}$ converges superlinearly to $v^{\ast}$. Moreover, if $H(v) - \nabla_{xx}^{2}L(v^{\ast}) = {\cal O}(\Vert v - v^{\ast} \Vert) ~ (v \to v^{\ast})$, then the sequence $\{ v_{k} \}$ converges quadratically to $v^{\ast}$.
\end{theorem}

\begin{proof}
Theorem~\ref{th:error_bound}, Lemma~\ref{lem:sigma_Lipschitz}, and Proposition~\ref{pro:solvable} guarantee the existence of constants $r_{1} > 0$ and $c_{1} > 0$ such that for each $v \in B(v^{\ast}, r_{1})$, problem {\rm P}$(v)$ has a local optimum $\omega(v) = (\xi(v), \zeta(v), \Sigma(v)) \in B(\omega^{\ast}, \nu)$ and
\begin{align} \label{ineq:three_ineq}
\Vert v - v^{\ast} \Vert \leq c_{1} \sigma(v), ~ \sigma(v) \leq c_{1} \Vert v - v^{\ast} \Vert, ~ \Vert v + \delta(v) - v^{\ast} \Vert \leq c_{1} \sigma(v),
\end{align}
where $v = (x, y, Z)$ and $\delta(v) = (\xi(v), \zeta(v)-y, \Sigma(v)-Z)$. Now we take $\varepsilon \in (0, \frac{1}{2})$ arbitrarily. The continuity of $\sigma \colon \V \to \R$ and item~(a) of Lemma~\ref{lem:sigma_property} ensure that there exists $r_{2} > 0$ satisfying
\begin{align} \label{ineq:sigma_eps}
\sigma(v) \leq \frac{r_{1}}{c_{1}}, \quad \sigma(v + \delta(v)) \leq \frac{\varepsilon}{c_{1}^{2}} \sigma(v)
\end{align}
for $v \in B(v^{\ast}, r_{2})$. Let $r_{3} \coloneqq \min \{ r_{1}, r_{2} \} > 0$. In what follows, we show that
\begin{align} \label{imply:well-defined}
v \in B(v^{\ast}, r_{3}) \quad \Longrightarrow \quad v + \delta(v) \in B(v^{\ast}, r_{3}).
\end{align}
By exploiting the third inequality in \eqref{ineq:three_ineq} and the first inequality in \eqref{ineq:sigma_eps}, we have $\Vert v + \delta(v) - v^{\ast} \Vert \leq c_{1} \sigma(v) \leq r_{1}$, namely, $v + \delta(v) \in B(v^{\ast}, r_{1})$ for $v \in B(v^{\ast}, r_{3})$. It then follows from the first and second inequalities in \eqref{ineq:three_ineq} and the second inequality in \eqref{ineq:sigma_eps} that
\begin{align} 
\Vert v + \delta(v) - v^{\ast} \Vert \leq c_{1} \sigma(v + \delta(v)) \leq \frac{\varepsilon}{c_{1}} \sigma(v) \leq \varepsilon \Vert v - v^{\ast} \Vert \label{ineq:superlinear_converge} 
\end{align}
for any $v \in B(v^{\ast}, r_{3})$. Using $\varepsilon \in (0, \frac{1}{2})$ and $v \in B(v^{\ast}, r_{3})$ derives $\Vert v + \delta(v) - v^{\ast} \Vert \leq \frac{1}{2} r_{3}$. This fact indicates that \eqref{imply:well-defined} holds and problem {\rm P}$(v + \delta(v))$ also has a local optimum in $B(\omega^{\ast}, \nu)$. In other words, if the initial point of Algorithm~\ref{Local_SQSDP} satisfies $v_{0} \in B(v^{\ast}, r_{3})$, then Algorithm~\ref{Local_SQSDP} is well defined. Let $\{ v_{k} \} \subset B(v^{\ast}, r_{3})$ be an infinite sequence generated by Algorithm~\ref{Local_SQSDP}. It then follows from \eqref{ineq:superlinear_converge} and $\varepsilon \in (0, \frac{1}{2})$ that $\Vert v_{k} - v^{\ast} \Vert \leq \varepsilon \Vert v_{k-1} - v^{\ast} \Vert \leq \cdots \leq \varepsilon^{k} \Vert v_{0} - v^{\ast} \Vert \leq (\frac{1}{2})^{k} r_{3}$ for any $k \in \mathbb{N}$, that is, $v_{k} \to v^{\ast}$ as $k \to \infty$. Moreover, inequality~\eqref{ineq:superlinear_converge} shows $\Vert v_{k+1} - v^{\ast} \Vert \leq \varepsilon \Vert v_{k} - v^{\ast} \Vert$ for all $k \in \mathbb{N} \cup \{ 0 \}$, namely, 
\begin{align*}
\lim_{k \to \infty} \frac{\Vert v_{k+1} - v^{\ast} \Vert}{\Vert v_{k} - v^{\ast} \Vert} = 0.
\end{align*}
This implies that $\{ v_{k} \}$ converges superlinearly to $v^{\ast}$.
\par
In what follows, we show the latter part of this theorem. From item~(b) of Lemma~\ref{lem:sigma_property}, there exist $r_{4} > 0$ and $c_{2} > 0$ satisfying
\begin{align} \label{ineq:sigma_cons}
\sigma(v + \delta(v)) \leq c_{2} \sigma(v)^{2}
\end{align}
for each $v \in B(v^{\ast}, r_{4})$. Let us define $r \coloneqq \min \{ r_{1}, r_{4} \} > 0$ and $c \coloneqq c_{1}^{3} c_{2} > 0$. Since $v_{k} \to v^{\ast}$ as $k \to \infty$, there exists $\ell \in \mathbb{N}$ such that $v_{k} \in B(v^{\ast}, r)$ for all $k \geq \ell$. Hence, combining~\eqref{ineq:three_ineq} and~\eqref{ineq:sigma_cons} yields
\begin{align*}
\Vert v_{k+1} - v^{\ast} \Vert \leq c_{1} \sigma_{k+1} \leq c_{1} c_{2} \sigma_{k}^{2} \leq c \Vert v_{k} - v^{\ast} \Vert^{2} \quad \forall k \geq \ell,
\end{align*}
where note that $\sigma_{k} = \sigma(v_{k})$. Therefore, $\{ v_{k} \}$ converges quadratically to $v^{\ast}$.
\qed
\end{proof}

\section{Concluding remarks} \label{section:conclusion}
In this paper, we have proposed a locally convergent stabilized SQSDP method, which is given as Algorithm~\ref{Local_SQSDP}, for problem~\eqref{NSDP}. Algorithm~\ref{Local_SQSDP} is based on the stabilized SQP methods for NLP problems and iteratively solves a sequence of the stabilized QSDP subproblems to generate a search direction and new Lagrange multipliers. We have provided the local convergence analysis of Algorithm~\ref{Local_SQSDP}. As a result, we have proven that Algorithm~\ref{Local_SQSDP} converges superlinearly and quadratically to the KKT point under the SRCQ and SOSC.
%several mild assumptions, not including the strict complementarity condition, constraint nondegeneracy, and strong SOSC.
%In the numerical experiments, fast local convergence of Algorithm~\ref{Local_SQSDP} has been confirmed.
\par
Future work will be to prove fast local convergence of the stabilized SQSDP method under degenerate cases.
%Indeed, we could observe that Algorithm~\ref{Local_SQSDP} converges fast for the degenerate problems in the numerical experiments. Therefore, there is a possibility that its fast local convergence can be proven in such cases.

%\section*{Acknowledgements}
%The author would like to thank two anonymous referees for their valuable and constructive comments. 

%\bibliographystyle{elsarticle-num}
%\bibliography{references.bib}

\end{document}